\documentclass[12pt, a4paper, reqno, twoside, makeidx]{amsart}

\usepackage[margin=2.8cm, marginpar=2cm]{geometry}
\usepackage{verbatim}
\usepackage{graphicx}
\usepackage{color}
\usepackage{empheq}
\usepackage[font={small,it}]{caption}
\usepackage[all]{xy}
\usepackage{tikz-cd}
\usepackage{tikz}
\usetikzlibrary{calc}
\usepackage{array,booktabs,calc}
\usepackage{amsmath, amsthm, amssymb, amsfonts, amscd}
\usepackage[english]{babel}
\usepackage[all]{xy}
\usepackage{url}
\usepackage{hyperref}
\usepackage{mathtools}
\usepackage{scalerel}
\usepackage{faktor}

\theoremstyle{remark}
\newtheorem{remark}{Remark}[section]
\theoremstyle{plain}
\newtheorem{theorem}{Theorem}[section]
\newtheorem{lemma}[theorem]{Lemma}
\newtheorem{proposition}[theorem]{Proposition}
\newtheorem{definition}[theorem]{Definition}
\newtheorem{corollary}[theorem]{Corollary}

\theoremstyle{definition}
\newtheorem{example}[theorem]{Example}

\newcommand{\Z}{\mathbb{Z}}

\newcommand{\R}{\mathbb{R}}

\newcommand{\homol}[2]{\Z_{(#1)}^{#2}}

\newcommand\restr[2]{{
  \left.\kern-\nulldelimiterspace 
  #1 
  \vphantom{\big|} 
  \right|_{#2} 
  }}

\usepackage{cellspace}
\newcommand\cincludegraphics[2][]{\raisebox{-0.3\height}{\includegraphics[#1]{#2}}}

\author{Daniele Celoria}

\author{Naya Yerolemou}
\begin{document}

\title[A discrete Morse perspective on knot projections]{A discrete Morse perspective on knot projections and a generalised clock theorem}

\begin{abstract}
We obtain a simple and complete characterisation of which matchings on the Tait graph of a knot diagram induce a discrete Morse function (dMf) on $S^2$, extending a construction due to Cohen. We show these dMfs are in bijection with certain rooted spanning forests in the Tait graph. We use this to count the number of such dMfs with a closed formula involving the graph Laplacian.
We then simultaneously generalise Kauffman's Clock Theorem and Kenyon-Propp-Wilson's correspondence in two different directions; we first  prove that the image of the correspondence induces a bijection on perfect dMfs, then we show that all perfect matchings, subject to an admissibility condition, are related by a finite sequence of click and clock moves. 
Finally, we study and compare the matching and discrete Morse complexes associated to the Tait graph, in terms of partial Kauffman states, and provide some computations.

\end{abstract}

\maketitle

\section{Introduction}

Given a graph $G$ embedded in the $2$-sphere,  denote by $G^*$ its plane dual, and by $\Gamma(G)$ the plane graph obtained by overlaying the two graphs. The vertices of $\Gamma(G)$ are divided into those coming from the vertices of the two original graphs and those arising as the intersection of dual edges.
Fix a pair of vertices $v$ and $f$, one in $G$ and one in $G^*$. A celebrated result by Kenyon, Propp and Wilson \cite[Theorem 1]{kenyon1999trees}, known as the KPW correspondence, provides a map between spanning trees in $G$ and perfect matchings on the graph obtained from $\Gamma(G)$ by removing $v$ and $f$. This map is shown to be a bijection if $v$ and $f$ are ``adjacent'', meaning that they are opposite vertices in a square face of $\Gamma(G)$. Otherwise, the map is only injective. 
\bigskip

With a somewhat different point of view, in his \emph{Formal Knot Theory}~\cite{kauffman2006formal} Kauffman introduced what are now know as \emph{Kauffman states}; these are a special kind of bijections between the set of crossings of a knot diagram and its regions, after two adjacent ``forbidden'' regions are excluded. In his famous Clock Theorem, he then proved that all Kauffman states are related by a finite sequence of simple swaps, known as \emph{clock moves}.
Every knot diagram $D$ can be chequerboard coloured, thus producing two dual graphs, having as vertices the white (respectively black) regions, and crossings as edges. 

Loosely speaking, Kauffman states are in bijection with spanning trees in the black graph, as well as with perfect matchings in the (Tait) graph $\Gamma(D)$ given by overlaying the two coloured graphs (see \emph{e.g.}~\cite{cohen2010twisted}). In particular, the KPW correspondence reduces to Kauffman's bijection for the black graphs with the forbidden regions as adjacent ones.
\bigskip

More recently, Cohen \cite{cohen2012correspondence} showed how to associate a discrete Morse function --as defined by Forman \cite{forman1998morse}-- to a Kauffman state on a given knot diagram. More precisely, rather than an actual discrete Morse function, his result yields an equivalence class of such objects, which can be thought of as being perfect matchings on the balanced Tait graph (see Section \ref{sec:diagrams} for the definition) of the diagram. 
One of the aims of this paper is to extend Cohen's work bridging these graph theoretic properties and ideas stemming from knot theory. We tried to keep the exposition as self contained as possible, as well as accessible to people with graph or knot theoretic backgrounds.
\bigskip

Our starting point is to completely characterise the set of possible dMfs arising from a generalised version of Cohen's construction (Theorem~\ref{thm:dmfandstates}). This will allow us (Theorem~\ref{thm:dmfforest}) to prove the existence of a bijection between dMfs arising this way and rooted spanning  orthogonal forests in the black and white graphs of the diagram, induced by a generalised version of Kauffman states, called \emph{partial Kauffman states}. 

We use this bijection to first count perfect dMfs in Proposition~\ref{prop:countperfect}, then --using a symbolic version of the graph Laplacian for the two coloured graphs-- we give a formula (Proposition~\ref{prop:countdmfs}) to count all dMfs, and compute them for a simple infinite family of diagrams in terms of Fibonacci numbers.
\bigskip

We then turn to perfect admissible matchings on $\Gamma(D)$, where admissible just means that exactly one vertex of each colour is unmatched.

We introduce a set of two moves, called click path and click loop moves; we first prove (Theorem \ref{prop:clickpathmove}) that click path moves induce bijections between the sets of perfect dMfs on $S^2$ with different critical points.
This implies immediately that the image of the KPW correspondence only consists of perfect dMfs, regardless of the adjacency condition; furthermore (Corollary \ref{cor:alldmfsfromKPW}) every perfect dMf arises uniquely as the image of precisely one choice of spanning tree and one vertex of each colour.

We then prove a topological characterisation (Theorem \ref{prop:whichcomponents})
of the subgraphs induced by perfect admissible matchings on the black and white graph, as well as a combinatorial one
(Theorem \ref{prop:admissibleodd}) in terms of Jordan resolutions. 

With these results at hand we can state our main result, which is a simultaneous further generalisation of the Clock theorem and KPW's correspondence:
\begin{theorem}[Click-Clock]\label{thm:clickclockintro}
If the knot diagram $D$ is reduced, any two perfect admissible matchings on $\Gamma(D)$ are related by a finite sequence of click path, click loop and clock moves.
\end{theorem}
One easy consequence of this result is that two perfect dMfs induced by Cohen's construction can be transformed into one another by clock and click path moves – or only clock moves if they share the same critical points (this last part is the original Clock Theorem). 

The proof of the Click-Clock theorem is surprisingly not straightforward, and relies on several  seemingly unrelated graph-theoretic constructions. It is also noteworthy to point out that this generalisation of the Clock Theorem is independent from others that have appeared in the literature, such as Hine and K{\'a}lm{\'a}n's  \cite{hine2018clock} version for triangulated surfaces, and Zibrowius' extension to tangles  \cite{zibrowius2016heegaard}. 
\bigskip

Finally we introduce two simplicial complexes associated to the set of partial Kauffman states, the \emph{matching} and \emph{discrete Morse} complexes. We study some of the properties of these mysterious complexes, and provide some sample computer aided computations of their homologies.

It remains an open question how to extract more information from these complexes associated to $\Gamma(D)$ and whether this information can be related to interesting features of the knot.

\subsection*{Acknowledgements} DC is supported by the European Research Council (ERC) under the European Unions Horizon 2020 research and innovation programme
(grant agreement No 674978). NY is supported by The Alan Turing Institute under the EPSRC grant EP/N510129/1. The authors would like to thank Vidit Nanda and Agnese Barbensi for their useful feedback and support, and the anonymous referee for their helpful comments.

\section{Knot diagrams and Kauffman states}\label{sec:diagrams}

Consider a knot $K \subset S^3$ and a diagram $D$ for $K$; as this will not create confusion we will also denote by $D$ the projection of the diagram, \emph{i.e.}~the $4$-valent graph\footnote{To avoid degenerate cases, we will always assume that the diagrams have at least one crossing.} on $S^2$ obtained by disregarding the over\slash under information at all of the crossings. Choose an arc $a$ on $D$ --that is, an edge in the projection-- and call both $a$ and the two regions of $S^2 \setminus D$ that have $a$ as an edge \emph{forbidden}. The pair $(D,a)$ is usually called a \emph{marked diagram} for $K$.

\begin{definition}[\cite{kauffman2006formal}]
A \emph{Kauffman state} on $(D,a)$ is a bijection between the crossings of $D$ and the non-forbidden regions of $S^2 \setminus D$, such that each crossing $c$ is assigned to one of the (at most $4$) non-forbidden regions that are incident to $c$.
\end{definition}

A common way of representing a Kauffman state on $(D,a)$ is to mark the edge $a$, and specify the assigned region to a given crossing by placing a dot near it.
\begin{figure}[ht]
  \centering
\includegraphics[width=8cm]{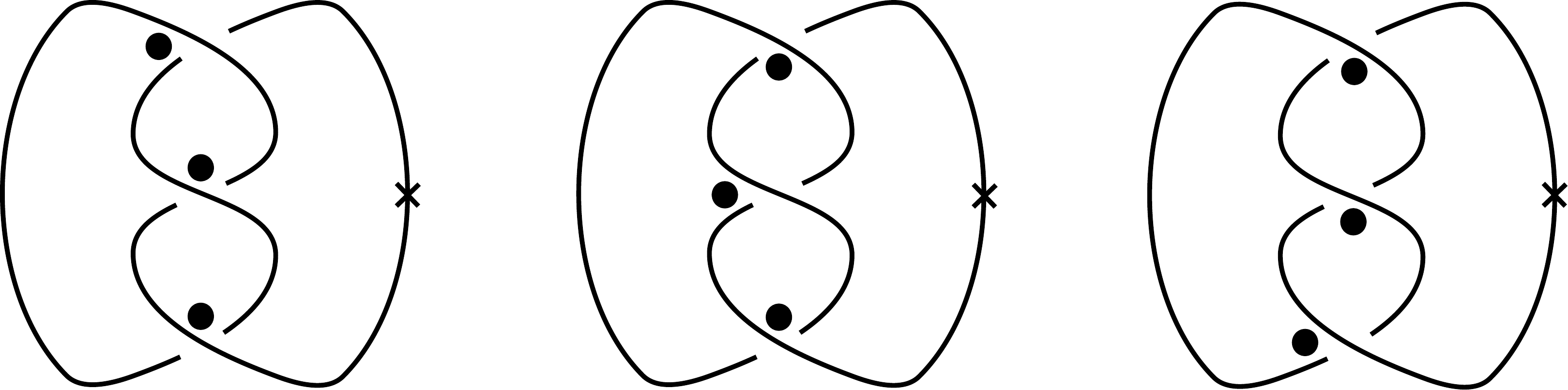}
\caption{The three Kauffman states on a minimal marked diagram $(D,a)$ of the left trefoil. The forbidden edge $a$ is marked with an asterisk.}
\label{fig:PKStrefoil}
\end{figure}

It was proved by Kauffman in~\cite{kauffman2006formal} that every pair of Kauffman states on $(D,a)$ is related by a finite sequence of \emph{clock moves}, shown in Figure~\ref{fig:clockmove}.

\begin{figure}[ht]
  \centering
\includegraphics[width=5cm]{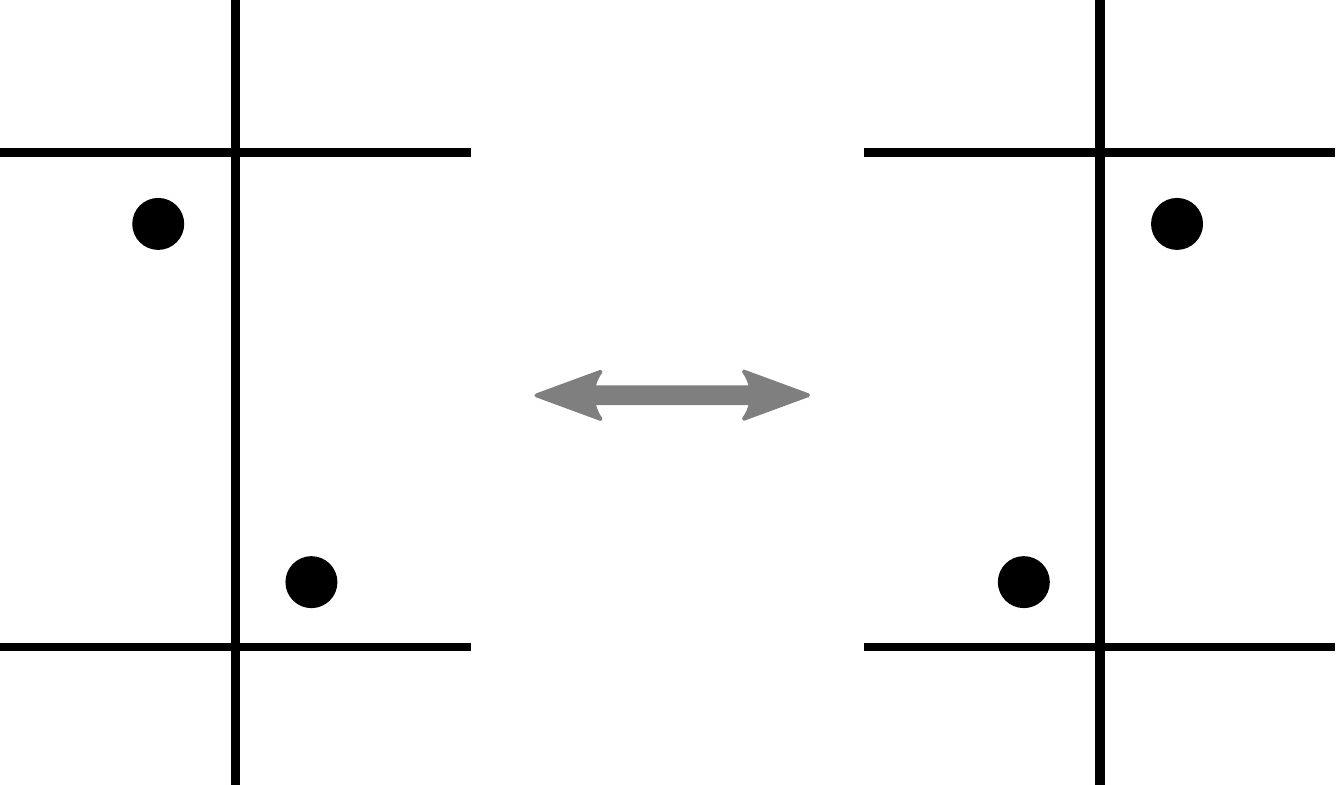}
\caption{A clock move seen on a local portion of a diagram. The right configuration is obtained from the left one by a clockwise clock move. The two Kauffman states are assumed to coincide everywhere except on the two crossings involved in the clock move.}
\label{fig:clockmove}
\end{figure}

So, given a marked diagram $(D,a)$, one can construct the \emph{clock graph} whose vertices are the Kauffman states $X(D,a)$, and edges are given by clock moves (see Figure~\ref{fig:tait}). Clock graphs have been studied previously in~\cite{abe2011clock}, \cite{cohen2014kauffman} and \cite{cohen2010twisted}, for example. The next objects we introduce have been previously studied in \cite{huggett_moffatt_virdee_2012} and \cite{cohen2010twisted}, for example.

\begin{figure}[ht]
  \centering
\includegraphics[width=8cm]{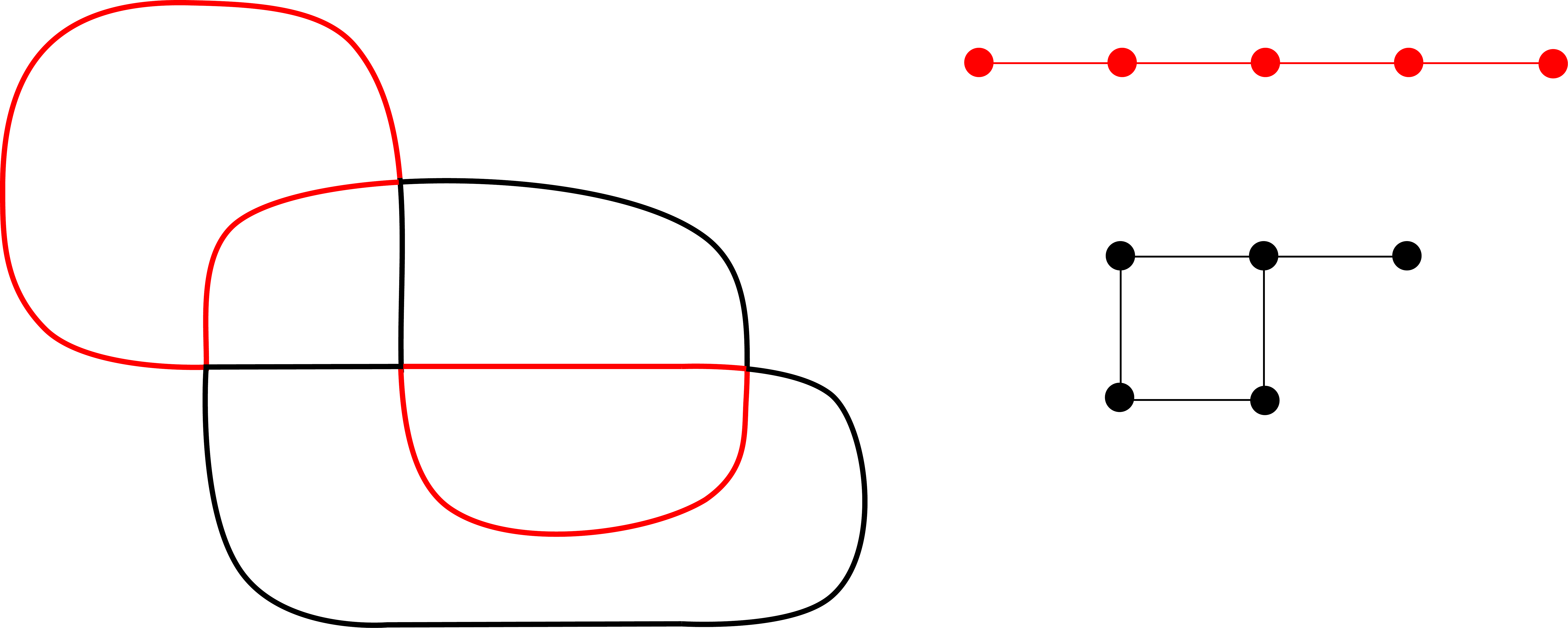}
\caption{The projection of a minimal diagram of the figure 8 knot. Choosing any of the the red arcs as forbidden results in the red clock graph on the right and similarly for any choice of black arc.}
\label{fig:clockgraphs}
\end{figure}

\begin{definition}
Let $D$ be a knot projection with a black and white chequerboard colouring \footnote{Note that this colouring always exists because $D$ is 4-valent.}. The black graph $G_b(D)$ has black regions as vertices and crossings as edges (the white graph $G_w(D)$ is defined analogously)\footnote{We will omit $D$ from the notation when it is clear to which diagram $G_b$ is associated.}. The \emph{overlaid Tait graph $\Gamma (D)$} is the superimposition of $G_b$ and $G_w$, using their natural embeddings in $S^2$.
\end{definition}

Note that the graphs $G_b$ and $G_w$ are plane duals, $\Gamma(D)$ divides the sphere into square regions, and 
$$V(\Gamma (D)) = [V(G_w) \cup V(G_b)] \cup [E(G_w) \cap E(G_b)],$$ and edges in $\Gamma (D)$ are the half edges of the white and black graphs, connecting one of their vertices to a crossing (see~\emph{e.g.}~\cite{kauffman2006formal} for a more detailed exposition). 

If one discards the forbidden vertices (one from $G_w$ and one from $G_b$), we obtain the  \emph{balanced overlaid Tait graph} $\Gamma (D,a)$. The vertices of this tri-partite graph can be split into those coming from black or white regions, and those corresponding to crossings. 

In the figures, we will draw black\slash white regions as black\slash white dots respectively, and vertices coming from crossings as little crosses, as in Figure~\ref{fig:tait}.

\begin{figure}[ht]
  \centering
\includegraphics[width=9cm]{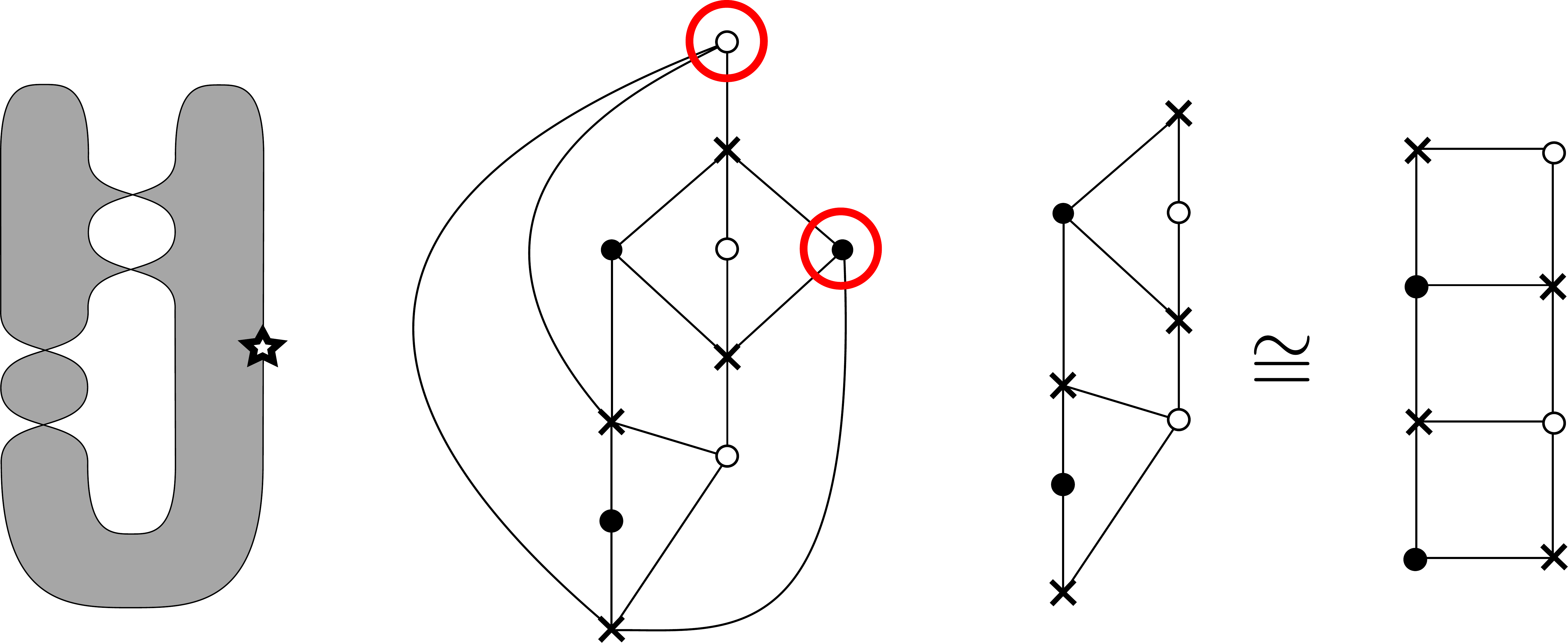}
\caption{From left to right: a chequerboard coloured projection for a diagram $D$ of the figure eight knot, the associated overlaid (unbalanced) Tait graph $\Gamma(D)$, and the balanced one $\Gamma(D,a)$ given by the choice of the arc $a$ marked with an asterisk. Forbidden regions are circled in red.}
\label{fig:tait}
\end{figure}

\begin{definition}[\cite{kauffman2006formal}]
The \emph{Jordan trail} associated to a Kauffman state $x\in X(D,a)$ is the curve formed from $x$ by resolving all crossings as indicated in Figure~\ref{fig:singleresolvecrossing}.
\end{definition}

\begin{figure}[ht]
  \centering
\includegraphics[width=4cm]{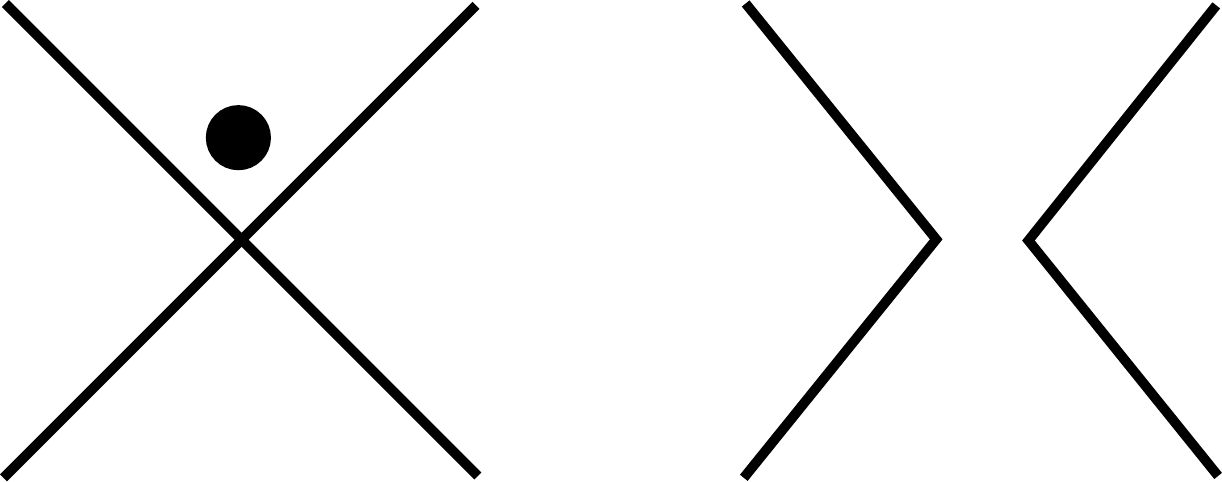}
\caption{How to resolve a crossing involved in a Kauffman state.}
\label{fig:singleresolvecrossing}
\end{figure}

We will see later on that Jordan trails associated to Kauffman states are always connected (see also Figure~\ref{fig:treekauffstate}). We will expand this definition in Section~\ref{sec:knottodmf} to include the cases where not all crossings are paired with a region (as in Figure~\ref{fig:partialresolvecrossing}). 
\bigskip

We recall here some graph-theoretic terminology that will be widely used throughout the rest of the paper. 
\begin{definition}
A \emph{matching} on a graph $G$ is a subset of edges in $G$ sharing no common vertices. A matching is \emph{perfect} if every vertex in $G$ is incident to an edge in the matching and it is \emph{maximal} if it is not a proper subset of another matching. A \emph{maximum matching} is a matching containing the largest number of edges. 
\end{definition}

We write $Tree(G)$ to denote the set of all spanning trees in $G$. Let $H$ be a subgraph of $G$, with  $E(G)=\{e_1,\ldots,e_m\}$ and $E(H)=\{e_i\}_{i \in I}$ for some $I \subseteq \{1,\ldots, m\}$. 
A graph $H^\perp$ in the dual graph $G^\ast$ \emph{orthogonal to $H$} is any subgraph of $G^\ast$ with edges a subset of $\{e_i^\ast\}_{i \notin I}$, where $e_i^\ast$ is the unique edge in $G^\ast$ intersecting $e_i$ in $G\cup G^\ast$. It is a well-known result that the dual $H^\perp$ of a spanning tree $H$ is itself a spanning tree in the dual graph, and that $H^\perp$ is completely determined by $H$. A forest is just a disjoint union of trees, and we allow isolated vertices to be connected components of a forest.

By definition, the set of perfect matchings of $\Gamma (D,a)$ is in bijective correspondence with $X(D,a)$  (see \emph{e.g.}~Figure \ref{fig:kauffstatesarematchings}), hence we will use these notions interchangeably throughout.

\begin{figure}[ht]
  \centering
\includegraphics[width=6cm]{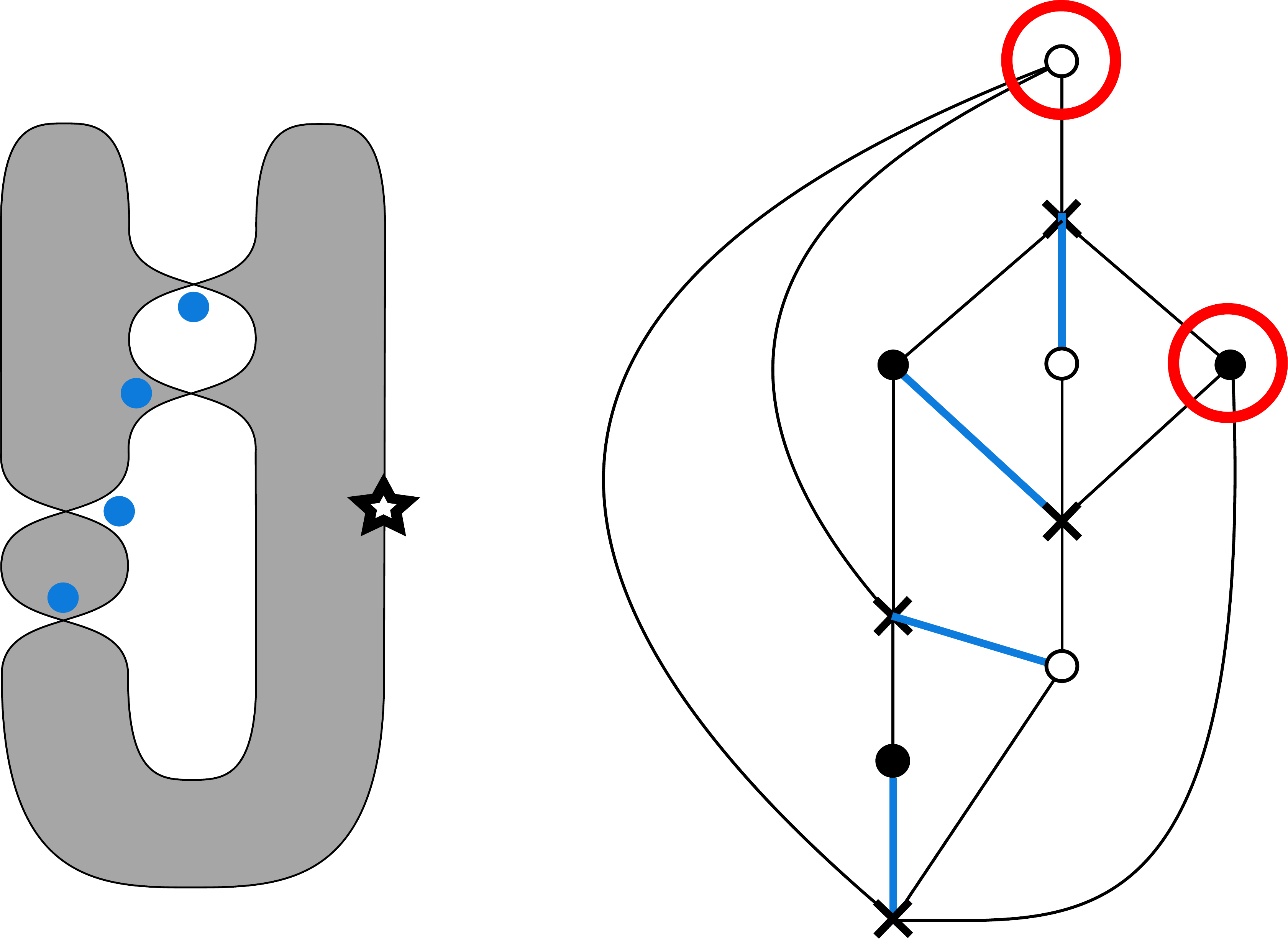}
\caption{On the left is a Kauffman state on a diagram for the figure eight knot and on the right is the corresponding perfect matching on the balanced Tait graph. The forbidden regions are circled in red.}
\label{fig:kauffstatesarematchings}
\end{figure}

There is also a bijection between Kauffman states on $(D,a)$ and rooted spanning trees in $G_b$ (or $G_w$), after choosing a common root for the trees and their duals (see \cite[Sec.~2]{kauffman2006formal}). The choice of these roots is dictated by the arc $a$: given a spanning tree $T \subset G_b$, take the unique vertex corresponding to a forbidden region as the root of the tree and orient all edges away from the root (see Figure \ref{fig:treekauffstate}). Let $T^*\subset G_w$ be the tree dual to $T$, and, again, orient the edges to flow away from the uniquely determined root. Each crossing lies at the centre of an oriented edge $u\to v$, so associate to each crossing the target $v$ of the edge to which it belongs. This will always be a non-forbidden region because the two forbidden regions are, by construction, only ever the source of an edge. There is a bijection between crossings and non-forbidden regions because $T$ and $T^*$ are dual and orthogonal, and they span all regions. In a similar way, one can construct a spanning tree from a Kauffman state. 

\begin{figure}[ht]
  \centering
\includegraphics[width=12cm]{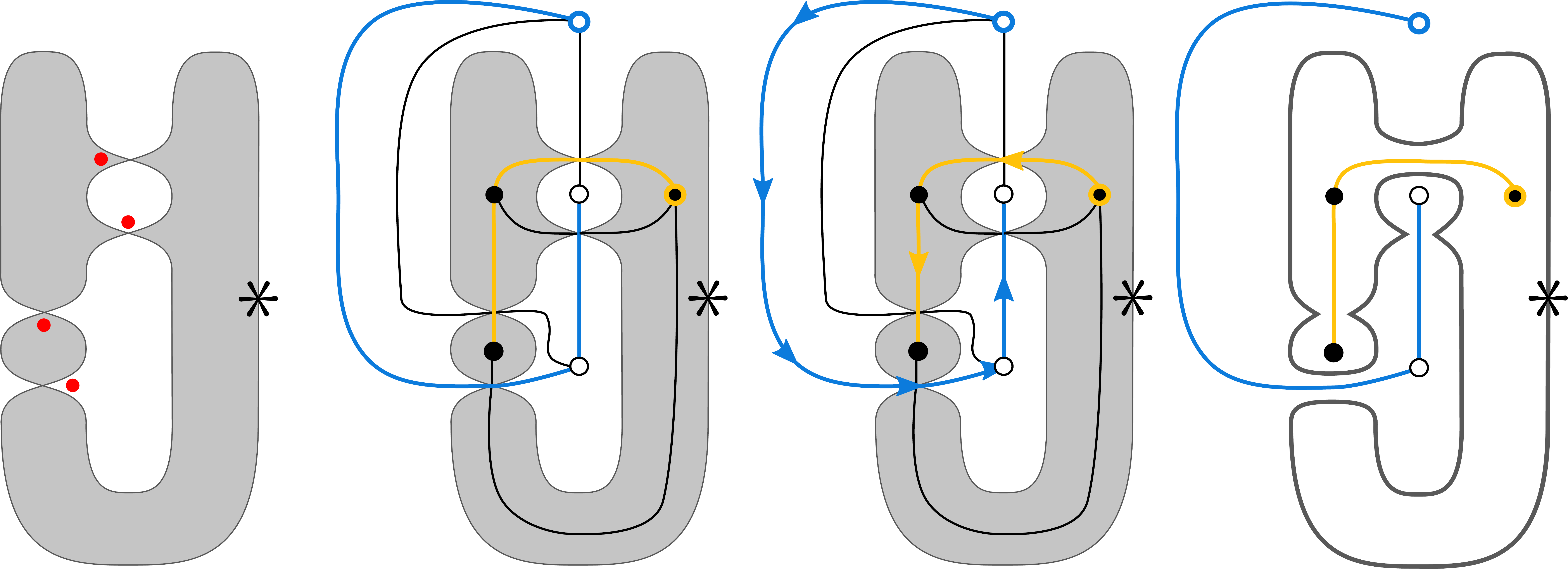}
\caption{Spanning trees determine Kauffman states: the Kauffman state on the left is completely determined by the oriented dual spanning trees in $G_b$ and $G_w$, rooted at the coloured nodes, and displayed in the two central parts of the figure. On the right, the Jordan resolution determined by the Kauffman state and the trees. Note that the Jordan resolution is a neighbourhood of both trees.}
\label{fig:treekauffstate}
\end{figure}

We can now state a version of Kauffman's  \emph{Clock Theorem} linking Kauffman states, appropriately rooted trees in the black or white graphs and Jordan trails, suited for our purposes.

\begin{theorem}{(Clock Theorem) \cite[Thm.~2.4]{kauffman2006formal}}\label{thm:jordantrails}
There is a bijection between $X(D,a)$, $Tree(G_b)$ and Jordan trails of $D$. Moreover, all elements of $X(D,a)$ can be transformed in one another with a finite sequence of clock moves.
\end{theorem}

Until now we only discussed the case where forbidden regions are determined by the choice of an arc on $D$. However (as also noted by Kauffman in \cite[Sec.~2]{kauffman2006formal}) this hypothesis can be relaxed. 
Denote by $(D,v_b,v_w)$ a diagram with exactly two forbidden regions, one black and one white, specified by the vertices $v_b \in G_b$ and $v_w \in G_w$, respectively and by $\Gamma(D,v_b,v_w)$ the subgraph of $\Gamma(D)$ with vertices $v_b$ and $v_w$ removed.

In~\cite[Thm. 1]{kenyon1999trees}, Kenyon, Propp and Wilson independently generalise the work by Kauffman to the case where the two unmatched vertices $v_b$ and $v_w$ are not necessarily adjacent. Here by adjacent we mean that the two corresponding regions share an edge in the projection; this is also equivalent to requiring that $v_b$ and $v_w$ are opposite vertices in a square face of the Tait graph $\Gamma(D)$. 

Their main result --recalled below-- known as the \emph{KPW correspondence} (or construction) is equivalent to the first part of Kauffman's Clock theorem, as stated in Section~\ref{sec:diagrams}, in the case of adjacent unmatched vertices. If instead the unmatched vertices are not at the opposite sides\footnote{In their setting, the matching obtained are all automatically admissible.} of a square in $\Gamma(D)$ we only get an injection between the rooted spanning trees of the black graph (with the extra choice of a vertex of the white graph) to perfect matchings in the corresponding balanced overlaid Tait graph.

\begin{theorem}[\cite{kenyon1999trees} Thm.~1]\label{thm:KPW}
Fix two vertices $v_b \in V(G_b)$ and $v_w \in V(G_w)$; then if these are adjacent, there is a bijective map from $Tree(G_b)$ to perfect matchings on $\Gamma(D,v_b,v_w)$. If instead the two vertices are not adjacent, the map is an injection.
\end{theorem}

The constructions and correspondences outlined here will be generalised in several different directions in the following  Sections. 

We will first extend the correspondence between perfect dMfs and rooted dual trees to non-maximal dMfs and \emph{orthogonal rooted forests}, then we will see in detail what happens when considering non-adjacent roots in the perfect setting. We will then prove that the image of the KPW correspondence  consists only of dMfs, and moreover, the elements in its image are in bijection with the set of perfect dMfs on a cell structure on $S^2$ induced by the given knot diagram.

What we have covered so far does not seem to be connected to clock moves, and hence to the second part of Kauffman's Clock theorem. We will remedy this in Section~\ref{sec:clickclock}. There we prove a direct generalisation of his theorem, extended to perfect admissible matchings on $\Gamma(D)$, after introducing two further moves in addition to the usual clock ones.

\section{Discrete Morse functions}\label{sec:dmf}

Discrete Morse functions, first introduced in~\cite{forman1998morse} by Forman, are functions assigning a real number to each cell in a regular cell complex, under certain combinatorial conditions. These discrete Morse functions, as the name suggests, can be thought of as being a discrete analogue of the ``classical'' smooth Morse functions.

Rather than dealing with functions themselves, following~\cite{chari2005complexes} we will use suitable equivalence classes, where two such functions are deemed equivalent if they induce the same \emph{acyclic partial matching}, which we define below. Following the previous analogy, these acyclic matchings should be considered as a discrete analogue of the gradient vector field of a Morse function. \bigskip

Let $X$ denote a regular cell complex. A \emph{partial matching} on $X$ is a decomposition of the cells of $X$ into three disjoint sets $R,U$ and $M$, along with a bijection $\mu:R\to U$ such that $dim(x)=dim\left(\mu(x)\right)-1$ and $x$ lies in the image of the attaching map of the cell $\mu(x)$, for all $x\in R$. Write $x<y$ if both aforementioned conditions hold for two cells $x$ and $y$. In the simplicial case, this just means that $x$ is a codimension 1 face of $\mu(x)$ for all $x\in R$. 

The bijection $\mu$ is \emph{acyclic} if the transitive closure\footnote{This is just the smallest transitive relation containing $\vartriangleleft$.} of the relation
\begin{equation*}
x \vartriangleleft x'\ \text{if and only if}\ x< \mu(x')
\end{equation*}
generates a partial order on $R$. The cells in $M$ are called \emph{critical}. A discrete Morse function is \emph{perfect} if $M$ is minimal. The following result is fundamental in discrete Morse theory.

\begin{theorem}{\cite[Cor.~3.5]{forman1998morse}}
If $\mu : R \to U$ is an acyclic partial matching on a regular CW complex $X$ with critical cells $M$, then $X$ is homotopy-equivalent to a CW complex whose $n$-dimensional cells correspond bijectively with the $n$-dimensional cells in $M$.
\end{theorem}

In particular, any perfect discrete Morse function on a cell complex on $S^2$ contains exactly two critical cells: one in dimension $0$ and one in dimension $2$. 

One can visualise discrete Morse functions on a \emph{poset graph} $\mathcal{H}(X)$ (Figure \ref{fig:hasse}), a representation of the poset of cells as a graph, oriented by the incidence relation $<$ defined above.

For cells $x<y$, there is a downward arrow from $y$ to $x$. Hence we can associate to a discrete Morse function the partial matching on $\mathcal{H}(X)$ composed of the edges involved in the relation $\vartriangleleft$.

\begin{figure}[ht]
  \centering
\includegraphics[width=12cm]{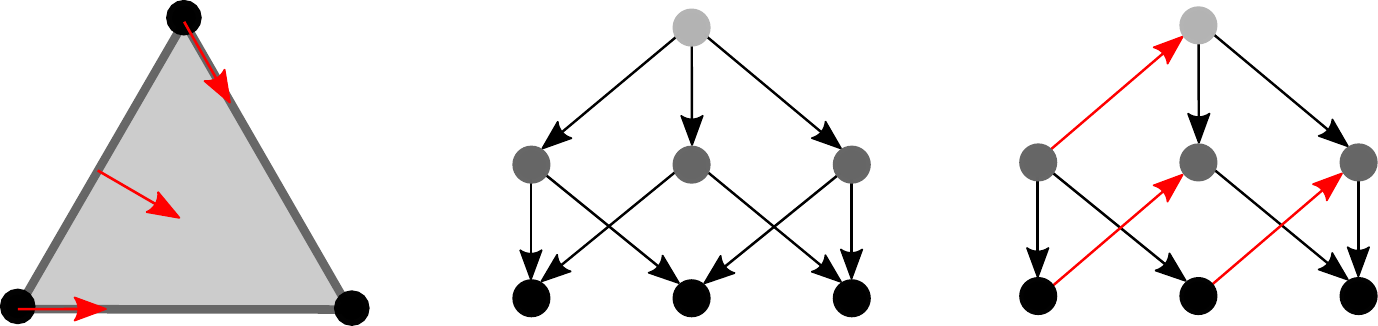}
\caption{A simplicial complex with an acyclic partial matching given by the red arrows. In the centre, its poset graph and on the right its amended poset graph, with arrows between matched cells reversed. We can see the matching is acyclic as there are no directed cycles in the amended poset graph diagram.}
\label{fig:hasse}
\end{figure}

\begin{remark}\label{rmk:dmfsubcomplex}
If $X$ is a cell complex and $\Sigma=\{(x,\mu(x))\}_{x\in R}$ an acyclic partial matching on $X$, then for any pair $Y=(y,\mu(y))\in\Sigma$, the matching $\Sigma\setminus Y$ is acyclic. This follows from the fact that since $\Sigma$ is acyclic, there is a partial order $\vartriangleleft$ on $R$, defined by $x \vartriangleleft x'$ if and only if  $x< \mu(x')$ and this is still a partial order on $R\setminus Y$.
\end{remark}

We will see in the next section that the overlaid Tait graph of a knot projection (with a suitable orientation) coincides with the poset graph of a certain cellular  structure induced on $S^2$.\\ 
We will sometimes abbreviate `equivalence class of a discrete Morse function' to `dMf' when referring to the equivalence class represented by a given acyclic matching.

\section{From knots to discrete Morse functions}\label{sec:knottodmf}

In this section we outline how to obtain matchings from a knot diagram and present the necessary conditions for such a matching to be acyclic. This is a generalisation of the construction presented by Cohen in~\cite{cohen2012correspondence}. We provide a proof for Cohen's assertion that a matching obtained from a Kauffman state via this construction is indeed a discrete Morse function in Proposition~\ref{prop:perfectdMf}.

Cohen outlines a correspondence in~\cite{cohen2012correspondence} and in~\cite[Sec. 8.2]{cohen2014kauffman}  between pairs $((D,a),x)$ with $x \in X(D,a)$ and discrete Morse functions on the $2$-sphere with a cell structure determined by $\Gamma(D,a)$. More explicitly, let us define the map
$$\Xi: \{\text{knot projections}\} \longrightarrow \{\text{cellular decompositions of }S^2\}$$
such that $\Xi(D)$ is the cellular structure on $S^2$ whose $2$-cells are given by the vertices of $G_w$, $1$-cells by crossings in $D$, and $0$-cells by the vertices of $G_b$. 

The unbalanced Tait graph associated to $(D,a)$ is a plane realisation of the abstract poset graph for this cellular structure; more precisely, we also need to orient the edges of $\Gamma(D)$ from white vertices to crossings, and from crossings to black vertices.  

As an example (for one choice of the colouring), the minimal diagrams of the alternating torus knots  $T_{2n+1,2}$ (see Figure \ref{fig:alternatingexample}) give the splitting of $S^2$ in two $(2n+1)$-gons, attached along their boundary. It follows from the definition that $G_b(D)$ provides the $1$-skeleton of $\Xi(D)$. 
\begin{figure}[ht]
  \centering
\includegraphics[width=13cm]{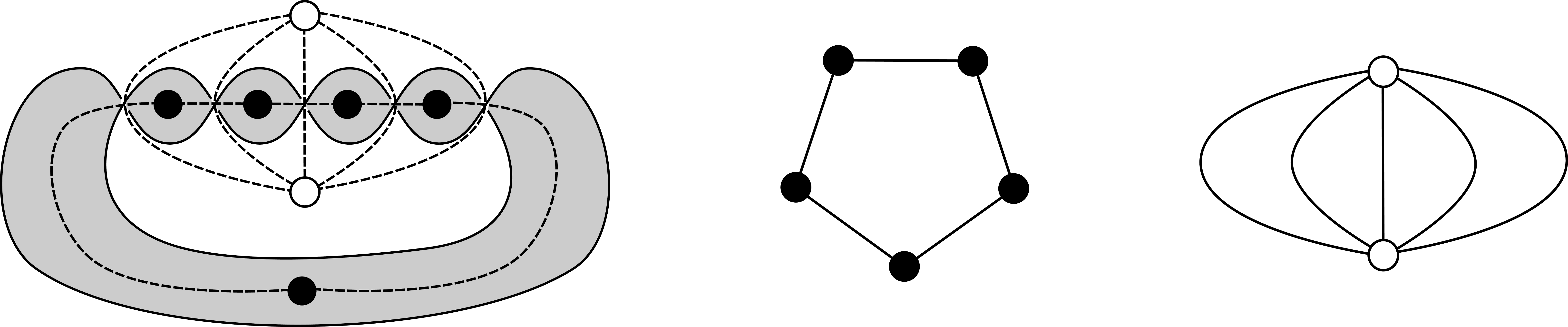}
\caption{The minimal projection $D_{2n+1}$ for the alternating torus knots $T_{2,2n+1}$ (for $n = 2$), together with its black and white graphs. The black graph, a cycle of length $2n+1$, provides the 1-skeleton of the cell structure on $S^2$ (which can be thought of as the equator). The two white vertices correspond to the two hemispheres, attached along their boundaries.}
\label{fig:alternatingexample}
\end{figure}

Cohen's main result from~\cite{cohen2012correspondence} is that if we ``extend'' $\Xi$ to pairs $((D,a),x)$ with $x \in X(D,a)$, then we get a discrete Morse function $\mu$ on $S^2$ with the cellular decomposition $\Xi(D)$. That is, the matchings on $D$ induced by Kauffman states are mapped to acyclic matchings on the cellular structure $\Xi(D)$ on $S^2$. A Kauffman state $x\in X(D,a)$ induces a matching in $\mathcal{H}(D)$ as it is a bijection between crossings and non-forbidden regions (corresponding to crossings and black or white vertices in $\Gamma(D)$). It is not immediately obvious that this induced matching is acyclic.

Cohen's construction induces maximum matchings; it was noted in Section \ref{sec:dmf} that any perfect discrete Morse function on $S^2$ has exactly two critical (unmatched) cells: one in dimension $0$ and one in dimension $2$. The forbidden regions in $(D,a)$ correspond precisely with one $2$-cell (the white forbidden region) and one $0$-cell (the black forbidden region), with any Kauffman state pairing up all other cells in $\Xi(D)$.

Using discrete Morse functions induced by Kauffman states gives rise to a restricted case of perfect discrete Morse functions, as the only two critical cells are incident to each other. 

We define\footnote{We acknowledge that the term has also been used to describe the generators of certain knot Floer homology complexes (see \emph{e.g.}~\cite{ozsvath2018kauffman}), but there is no chance of confusion with the ones defined below.} \emph{partial Kauffman states} to encompass non-maximum matchings, as well as maximum matchings with non-adjacent critical cells. 
\begin{definition}
Consider a knot diagram $D$ with $n$ crossings $\{c_1, \ldots, c_n\}$. The set of \emph{partial Kauffman states } $\widetilde{X}(D)$ is the set of injections  from a subset $\{c_i\}_{i \in I}$ to $V(G_w) \sqcup V(G_b)$, such that each crossing is paired with exactly one of its four adjacent regions, for all subsets $I \subset \{1,\ldots,n\}$. 
\end{definition}

A partial Kauffman state (pKs) on the crossings  indexed by $I$ will be often represented by placing one dot --called a \emph{component} of the pKs-- near each crossing in $\{c_i\}_{i \in I}$, in one of the four regions incident to it. In the same way, if $v_b \in V(G_b)$ and $v_w \in V(G_w)$, we can define $\widetilde{X}(D,v_b,v_w)$ as the set of pKs that do not have any component in the regions $v_b$ and $v_w$. If, moreover, these two regions are adjacent, we will shorten the notation to $\widetilde{X}(D,a)$, where $a$ is the unique arc separating $v_b$ and $v_w$.

A partial Kauffman state $x$ is \emph{admissible} if there is at least one unpaired region of each colour. $x$ is \emph{maximal} if it injects from the entire set of crossings of the diagram $D$. 

\begin{definition}
The \emph{Jordan resolution} $J(x)$ of a partial Kauffman state $x\in\widetilde{X}(D)$ is the curve formed by resolving all dotted crossings as shown in Figure~\ref{fig:singleresolvecrossing}. It is the analogue of the Jordan trail of a Kauffman state; one example is displayed in Figure~\ref{fig:partialresolvecrossing}. Write $|J(x)|$ to denote the number of connected components of the Jordan resolution, which we will call \emph{Jordan cycles}.
\end{definition}

Since $x$ is not necessarily maximal, $J(x)$ contains double points at all crossings not paired in $x$. Any connected, complete (\emph{i.e.~}all crossings are resolved) Jordan resolution is just a Jordan trail.

\begin{figure}[ht]
  \centering
\includegraphics[width=11cm]{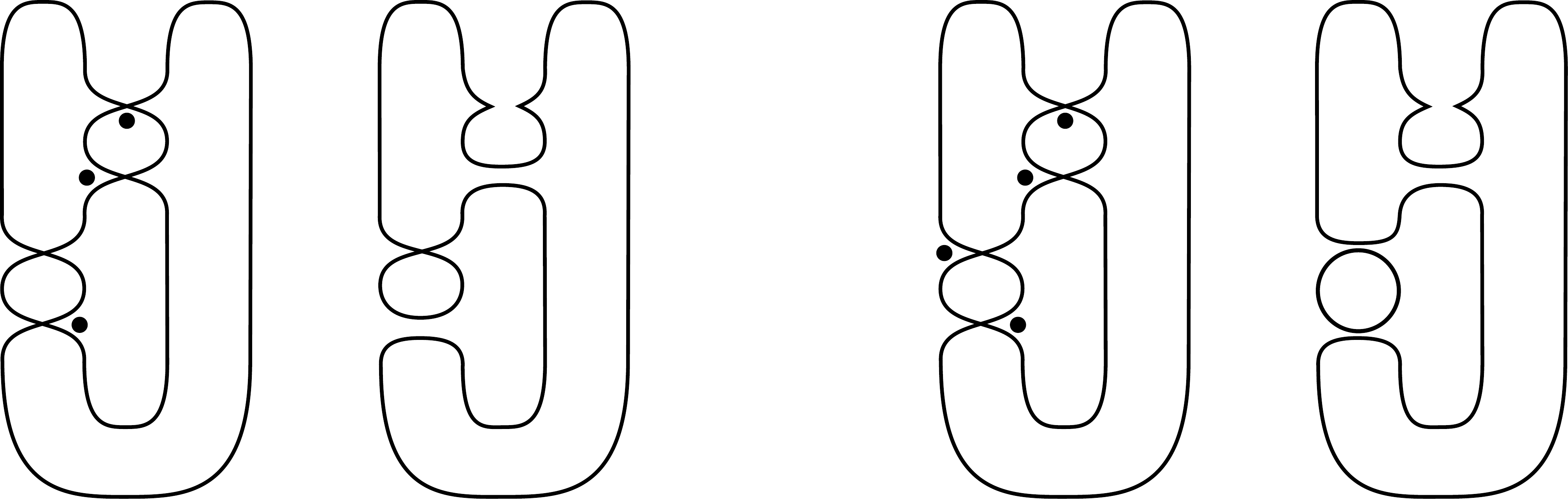}
\caption{Two partial Kauffman states and their Jordan resolutions; on the left, a connected partial resolution and on the right, a (disconnected) Jordan resolution. The pKs on the right is perfect but has non-adjacent unmatched regions; further it is not admissible.}
\label{fig:partialresolvecrossing}
\end{figure}

\begin{figure}[ht]
  \centering
\centering
\includegraphics[width =2cm]{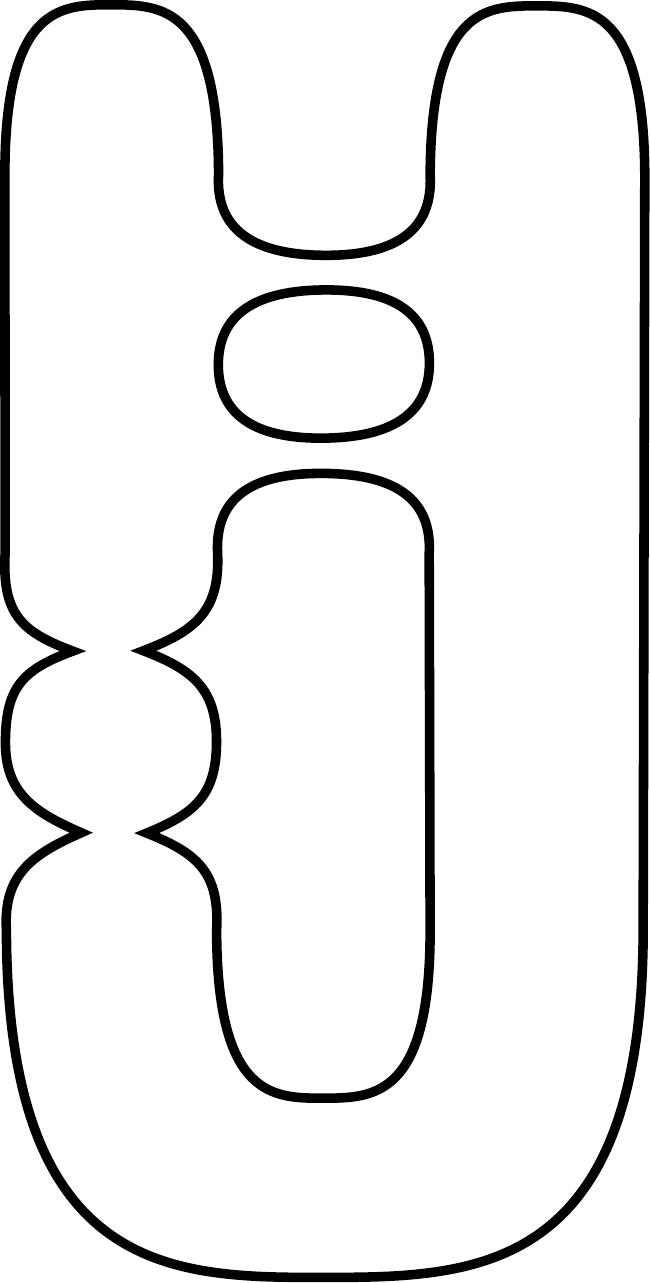}
\caption{A complete resolution not induced by a perfect matching. In particular, there is no pKs inducing this resolution. Note that the circles are not concentric on $S^2$ (\emph{cf.}~Proposition~\ref{prop:admissibleodd}).}
\label{fig:bad}
\end{figure}
We point out \emph{en passant} that, for a projection with $n$ crossings, there are a priori $2^n$ possible complete resolutions. However only a small number of these are Jordan resolutions; Figure~\ref{fig:bad} shows a complete resolution which is not induced by any perfect matching on $\Gamma(D)$. 

We say a simple loop $\gamma$ in $\Gamma(D)$ is \emph{monochromatic} if its vertices consist only of crossings and black regions or only of crossings and white regions. We say that a matching $x \in \widetilde{X}(D)$ \emph{supports} the monochromatic loop $\gamma$, if the edges of $\gamma$ are alternatively composed by edges of $x$. One example is shown in Figure~\ref{fig:isolatedcritical}.

\begin{figure}[ht]
  \centering
\includegraphics[width=9cm]{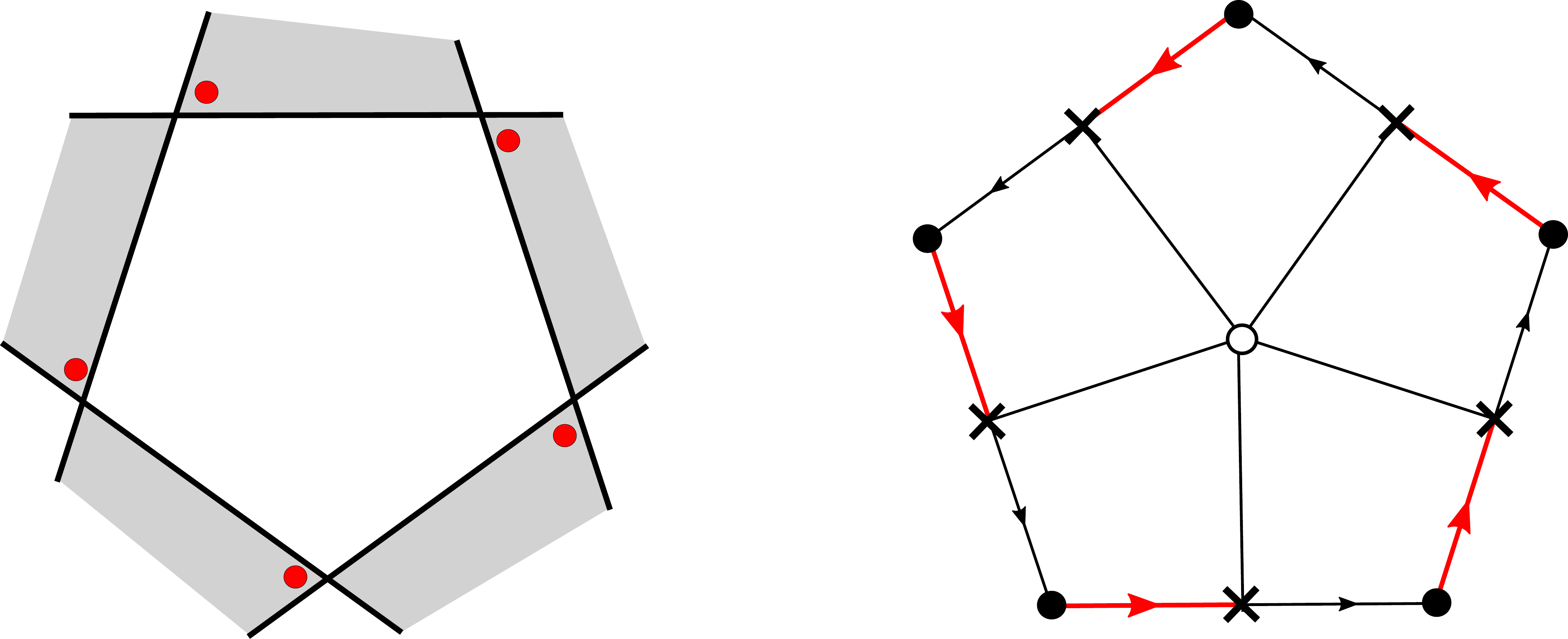}
\caption{An isolated critical point can give rise to partial matchings not inducing discrete Morse functions. The central white vertex is \emph{isolated} because all adjacent crossings are matched with other vertices. The red dots\slash edges represent a partial Kauffman state supporting the black loop surrounding the white $2$-cell.}
\label{fig:isolatedcritical}
\end{figure}

The next result is the main tool for checking whether a matching induced by a partial Kauffman state is a discrete Morse function or not.
\begin{theorem}\label{thm:dmfandstates}
A partial Kauffman state $x$ on $\Gamma(D)$ induces a discrete Morse function on $\Xi(D)$ if and only if $x$ does not support a monochromatic loop.
\end{theorem}
\begin{proof}
It is easy to see that if $x$ supports a monochromatic loop, then its components on the loop give rise to an oriented cycle in the poset\slash Tait graph $\Gamma(D)$, as in Figure~\ref{fig:isolatedcritical}. Conversely, if $x$ does not support any monochromatic loop, then all possible oriented loops in $\Gamma(D)$ must contains vertices of both colours. Hence, in each directed cycle in $\Gamma(D)$ we can find a portion of $\Gamma(D)$ that looks like one of the two configurations in Figure~\ref{fig:monoloops}. 
\begin{figure}[ht]
  \centering
\includegraphics[width=4cm]{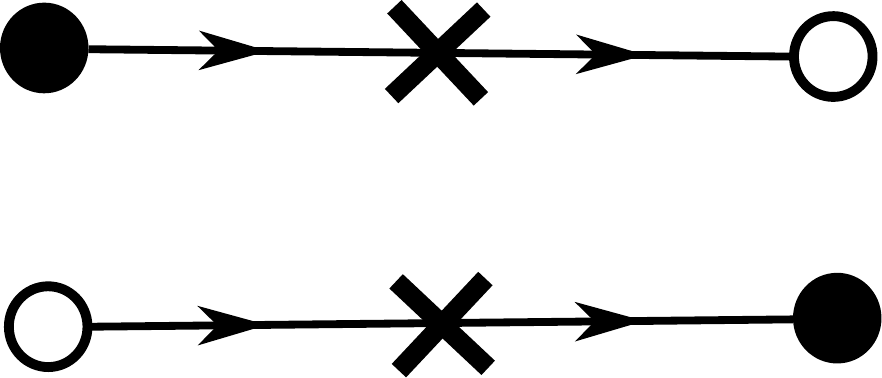}
\caption{Any directed cycle (induced by a partial Kauffman state) in the Tait diagram containing vertices of both colours must contain these configurations.}
\label{fig:monoloops}
\end{figure}
The top configuration can occur if and only if both vertices have been matched with the central crossing, as both arrows have been reversed, which contradicts the definition of a partial Kauffman state. The bottom configuration can occur in a directed cycle if and only if there is (eventually) a section of the cycle containing the top configuration (as there would have to be an arrow from a black vertex to a crossing and another from the crossing back to the white vertex on the left). Hence, there are no such directed cycles.
\end{proof}

The next result is an extended version of the the construction given by Cohen in \cite{cohen2012correspondence}.

\begin{proposition}\label{prop:perfectdMf}
Assume $x$ is a maximal pKs on $D$. If the unpaired white and black vertices are two vertices of a square in $\Gamma(D)$, \emph{i.e.}~the two corresponding regions are adjacent in the diagram, then $x$ induces a perfect discrete Morse function on $\Xi(D)$.
\end{proposition}
\begin{proof}
By Theorem~\ref{thm:dmfandstates}, it is sufficient to show that a maximal partial Kauffman state $x$ with adjacent unpaired vertices does not admit a monochromatic loop on $\Gamma(D,a)$, where $a$ is the unique arc common to the two unpaired regions.  Assume, for a contradiction, that $x$ contains a monochromatic loop $\gamma$. The loop divides $S^2$, and consequently the vertices of  $\Gamma(D,a)$, in to three parts; those belonging to $\gamma$, and two other connected components $R$ and $R^\prime$. We can assume without loss of generality that the two forbidden regions are contained in $R$, as they are adjacent and so cannot be separated by $\gamma$.
The thesis will then follow from the Lemma below.
\begin{lemma}\label{lem:unpaired}
Let $x$ be a perfect and admissible pKs on $\Gamma(D)$, supporting a monochromatic loop $\gamma$. Then, both regions in the complement of $\gamma$ contain exactly one unpaired vertex. 
\end{lemma}
\begin{proof}
Again, call the two regions $R$ and $R^\prime$, and assume wlog that $\gamma$ is black and has length $2n$, meaning that there are $n$ black regions and $n$ crossings as vertices. Consider the cellular structure on $R$ induced by the restriction $G_b^\prime$ of the black graph to $R$. Since $R$ is a disc, we have $1 = V(G_b^\prime) - E(G_b^\prime) + F(G_b^\prime)$, where $F(G_b^\prime)=V(G_w^\prime)$ since the graphs are dual. 

We can write $ V(G_b^\prime) = n + b^\circ$, $E(G_b^\prime) = n + c^\circ$, $F(G_b^\prime) = w^\circ$, where $b^\circ, c^\circ$ and $w^\circ$ are the number of black vertices, crossings and white vertices, respectively, in the interior of $R$. Then, $1+c^\circ = b^\circ + w^\circ$, and there is at least one unpaired white\slash black vertex. Since we are assuming that $x$ is maximal, there can be at most one unpaired vertex in each region.
\end{proof}
We can now conclude the Proposition \ref{prop:perfectdMf} by noting that, by assumption, $R$ contains two unpaired vertices.
\end{proof}

\begin{figure}[ht]
  \centering
\includegraphics[width=12cm]{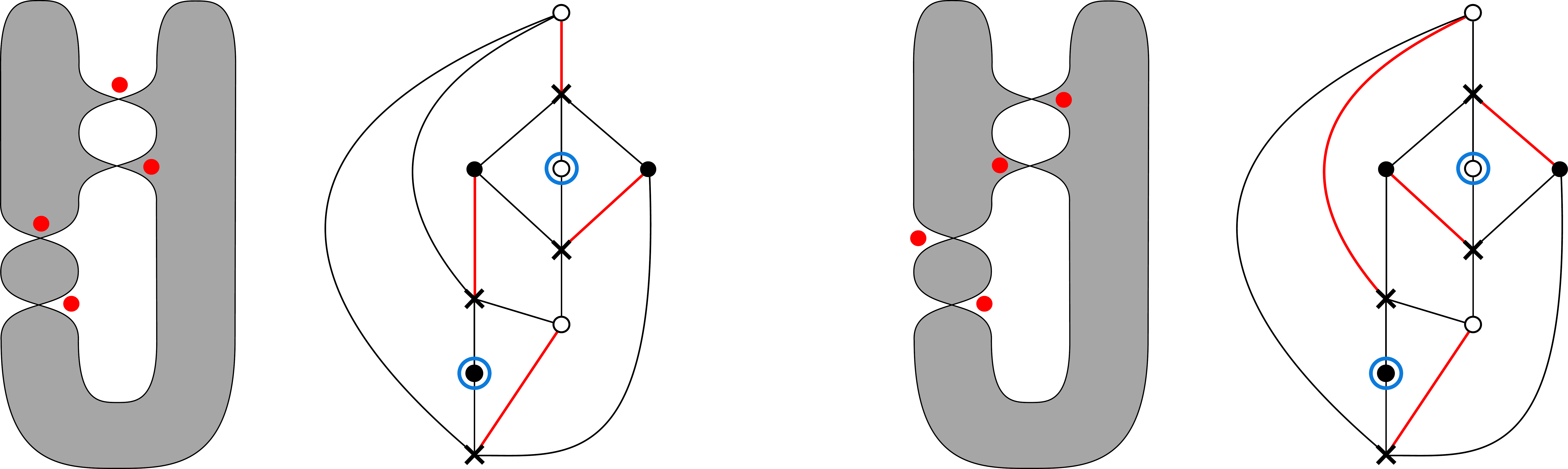}
\caption{On the left, a maximal matching inducing a dMf for a projection of the projection from Figure~\ref{fig:tait}, with non-adjacent unpaired black and white regions (circled in blue).  The red pKs shown on the right is maximal too, but contains two monochromatic loops, one white and one black. Hence, it does not induce any discrete Morse function.}
\label{fig:nodmf}
\end{figure}

\begin{remark}\label{rmk:critpoints}
If a matching admits a monochromatic loop $\gamma$, then it follows from the proof of Proposition~\ref{prop:perfectdMf} that $\gamma$ necessarily contains at least one unpaired vertex on both regions it divides $S^2$ in. It is however easy to exhibit examples of knot diagrams admitting perfect matchings with an arbitrary number of monochromatic loops.
\end{remark}

We conclude this section with a result relating discrete Morse functions to spanning forests, which should be thought of as an analogue of \cite[Prop~3.1]{chari2005complexes} and \cite[Thm~1.17]{zibrowius2016heegaard}, and as a generalisation of the already mentioned (see also~\cite{kenyon1999trees} and~\cite{kauffman2006formal}) bijection between Kauffman states and orthogonal pairs of rooted spanning trees in $G_b$ and $G_w$.

\begin{theorem}\label{thm:dmfforest}
There is a bijection between discrete Morse functions on $\Xi(D)$ induced by admissible partial Kauffman states and rooted orthogonal spanning forests in $G_b$ and $G_w$. In particular, when the admissible partial Kauffman states are maximal, this reduces to a bijection between perfect discrete Morse functions and rooted orthogonal spanning trees.
\end{theorem}

The latter part of the statement is just a part of Kauffman's Clock Theorem for non-necessarily adjacent forbidden regions. Before we can prove the theorem, we need to introduce a version of the trees induced by a Kauffman state, extended to admissible partial states.
\begin{definition}
Given a partial Kauffman state $x$ inducing a discrete Morse function on $\Xi(D)$ we can associate to it a pair of subgraphs $H_x^b,H_x^w$ of $G_w$ and $G_b$ respectively, in the same way as is done with perfect discrete Morse function: if $x$  matches a $0\slash 1$-cell pair, we include in $H_x^b$ the edge of $G_b$ corresponding to the crossing giving the matched $1$-cell in $\Xi(D)$. Likewise with $1\slash 2$-cell pairs and edges in $H_x^w$.
\end{definition}

\begin{proof}[Proof of Theorem~\ref{thm:dmfforest}]
First, the subgraphs $H_x^b$ and $H_x^w$ associated to $x$ are in fact forests as a consequence of Theorem~\ref{thm:dmfandstates}. Suppose that $H_x^b$ and $H_x^w$ are not orthogonal; then there exists an edge $e$ in $H_x^b$ intersecting an edge $e^\ast$ in $H_x^w$. But $e$ and $e^\ast$ intersect at a unique crossing of $D$, and for both edges to be included in their respective graphs, the crossing must be matched with one black and one white region, contradicting the injectivity condition of a partial Kauffman state. Hence, $H_x^b$ and $H_x^w$ are orthogonal.

In this setting, critical points of $x$ are given by unmatched vertices and crossings; each connected component (which is necessarily a  tree --possibly composed by a single vertex), contains exactly one unmatched region. The root of each connected component in the forests is the unique unmatched vertex in each component.

Conversely, any such pair of spanning orthogonal forests produces a unique equivalence class of dMfs by applying the process in reverse; there are no induced directed cycles in $\Gamma(D)$ as the matching is induced by forests, hence it gives an equivalence class of dMfs. The roots are uniquely determined in a similar way to the description in Figure \ref{fig:treekauffstate}.

When we restrict to maximum pKs, there is exactly one unmatched vertex of each colour, corresponding to exactly one root for each forest, hence the forests contain exactly one connected component each. 
\end{proof}

We conclude this section with the following result, relating Jordan resolutions to dMfs.

\begin{corollary}\label{cor:dmfconnectedJ}
An admissible partial Kauffman state $x \in \widetilde{X}(D)$ induces a discrete Morse function on $S^2$ (through Cohen's construction) if and only if $J(x)$ is connected.
\end{corollary}
\begin{proof}
The result follows easily from the characterisation provided of matchings induced by dMfs just provided, together with the fact that monochromatic loops always disconnect $J(x)$.
\end{proof}

\section{Counting discrete Morse functions}\label{sec:counting}

Theorem~\ref{thm:dmfforest} provides us with enough structure to count the number of discrete Morse functions on $\Xi(D)$, for a given knot diagram $D$. Figuring out the number of (perfect) dMfs for a given class of simplicial complexes is in general a challenging problem, and the precise number is known, for example, for the complete graph $K_n$, but unknown for the $n$-simplex for $n > 3$ (see respectively Sections $3$ and $5$ of~\cite{chari2000discrete}).\bigskip

The \emph{Laplacian} $L(G)$ of a graph $G$ on $n$ vertices is the $n\times n$ matrix
$$L(G) = \text{D}(G)-\text{A}(G),$$
where $\text{D}(G)$ is the diagonal matrix with vertex degrees on the diagonal and $\text{A}(G)$ is the adjacency matrix of $G$. It is well-known\footnote{See~\cite{biggs1993algebraic} for a complete source of related definitions and results.} that for a connected graph, the Laplacian has exactly one 0 eigenvalue, and that the non-zero eigenvalues $\lambda_2\leq\cdots\leq\lambda_n$ are strictly positive. Denote the characteristic polynomial of $L(G)$ by
$$p_{G} (t) = t^n + c_1 t^{n-1} + \ldots + c_{n-1}t + c_n.$$
Since we are dealing only with connected graphs, we can assume that $c_n = 0$. To stress the dependence of the coefficients on the graph, we will sometimes write $c_i(G)$ instead of just $c_i$.

Given an integer $0\le e \le |E(G)|$, let $\phi(G,e)$ denote the number of spanning forests in $G$ with exactly $e$ edges and let $F_i^e$ denote a subforest of $G$ with $e$ edges, for $i = 1, \ldots, \phi(G,e)$. For each $F_i^e$, define
$$\rho(F_i^e) = m_1\ldots m_k,$$ where $m_j$ is the number of vertices in the $j$-\emph{th} connected component of $F_i^e$, so $\rho(F_i^e)$ counts the number of distinct ways of rooting the forest $F_i^e$. It is a well-known result~\cite[Thm.~7.5]{biggs1993algebraic} that
\begin{equation}\label{eqn:forestcoeff}
(-1)^e c_e(G) = \sum_{i = 1}^{\phi(G,e)} \rho(F_i^e).
\end{equation}
In other words, the $e$-\emph{th} coefficient of $p_{G}$ is (up to alternating sign) the number of rooted forests in $G$ with $e$ edges. 
In what follows, $|det(K)|$ denotes the \emph{knot determinant}, which  is defined as the evaluation at $-1$ of the Alexander polynomial of $K$ \cite{rolfsen2003knots}.

\begin{proposition}\label{prop:countperfect}
Let $D$ be a diagram of a knot $K$ and $L_b$ the Laplacian matrix of $G_b(D)$. Let $\lambda_2,...,\lambda_{n}$ be the non-zero eigenvalues of $L_b$, where $n = |V(G_b)|$. There are exactly 
$$ \left(\prod_{i=2}^{n}\lambda_i \right)\cdot |V(G_w)| $$ 
perfect discrete Morse functions on $\Xi(D)$. Moreover, if $D$ is a minimal diagram of an alternating knot $K$, then the expression can be written as $$|det(K)| \cdot |V(G_b)|\cdot |V(G_w)|.$$
\end{proposition}
\begin{proof}
By Kirchhoff's Matrix-Tree Theorem, there are $$\frac{1}{n}\prod_{i=2}^{n}\lambda_i$$ unrooted spanning trees in $G_b$. For each such tree, we have $|V(G_b)|$ choices for the root in $G_b$ and $|V(G_w)|$ for the choice of root in $G_w$, yielding the desired expression. 

If $D$ is a diagram of an alternating knot, then $|det(K)|$ coincides with the number of (unrooted) spanning trees in either $G_b$ or 
$G_w$~\cite{BZknots}. 
\end{proof}

In addition to enumerating the possible spanning trees, we can explicitly list them using the \emph{symbolic Laplacian matrix $L^{symb} (G)$}. $L^{symb}_{i,i}$ is the formal sum of the edges incident to the vertex $v_i$ and $L^{symb}_{i,j}$ is the negative formal sum of all edges connecting $v_i$ to $v_j$, for $i\neq j$. 

Let $p^{symb}_G(t, E(G))$ denote the characteristic polynomial $det(L^{symb} (G) - t\cdot \text{Id})$, which is in $\Z [t, e_1, \ldots, e_m]$, where $E(G)=\{e_1,\ldots,e_m\}$. The coefficient of $t^k$ is (up to an alternating sign) the formal sum of all possible spanning forests in $G$ with $m-k$ edges, counted with the possible ways of rooting them. In other words, each monomial in this coefficient consists of a formal product of the edges contained in a single forest, multiplied by the number of possible roots. 
\bigskip

Using the symbolic Laplacian for both the black and white graphs, we can count the total number of discrete Morse functions on $\Xi(D)$:
\begin{proposition}\label{prop:countdmfs}
Let $\#\mathcal{M}(D)$ denote the number of dMfs on $\Xi(D)$. Then we can compute $\#\mathcal{M}(D)$ by considering the product
\begin{equation}\label{eqn:counting}
p_{G_b^*}^{symb} (-1,E(G_b)) \cdot p_{G_b}^{symb} (-1,E(G_b^*)) \in \faktor{\Z\left[ E(G_b),E(G_b^*)\right]}{\langle e_i \cdot e_i^*\rangle_{i = 1, \ldots , |E(G_b)|}}
\end{equation}
and evaluating all variables $e_i$ and $e_i^*$ in $1$ for $i = 1, \ldots, m$ for a minimal representative.
\end{proposition}

Note that the evaluation of the polynomial in Equation \eqref{eqn:counting} is just the characteristic polynomial of the symbolic Laplacian (up to a multiplication by $t$, which only affects the result up to a factor of $-1$) for the disjoint union of the graph $G_b$ and its dual; here by disjoint union we mean that the two are to be thought of as being disjointly embedded in the plane. 
An alternative to evaluating in this quotient ring
is to consider the product $p_{G_b^*}^{symb} (-1,E(G_b)) \cdot p_{G_b}^{symb} (-1,E(G_b))$ in the same $m$ variables, discard all non square-free monomials, and then set all variables $E(G_b)$ equal to $1$.

\begin{proof}[Proof of Proposition~\ref{prop:countdmfs}]
Let us start by noting that the coefficients of the characteristic polynomial of the Laplacian of a graph alternate in sign, by Equation \ref{eqn:forestcoeff}, so evaluating in $t= -1$ just gives the sum of the absolute values of the coefficients; moreover, by Theorem~\ref{thm:dmfforest} the number of dMfs on $\Xi(D)$ coincides with number of rooted orthogonal spanning forests in $G_b$ and its dual.
Using the symbolic Laplacian we see that for a plane graph $G$, the evaluation $p^{symb}_G(-1, E(G)) \in \Z[E(G)]$ is a sum of monomials, each of which gives a rooted spanning forests in $G$, and whose coefficient is the number of possible rootings. In particular $p_G(-1,1, \ldots,1)$, 
is the number of rooted spanning forests in $G$.

Now, the product in Equation \eqref{eqn:counting} gives a sum of monomials in the variables $E(G_b)$ and $E(G_b^*)$. The condition $e_i \cdot e_i^* = 0$ is equivalent to the orthogonality of the forests.
\end{proof}

\begin{example}
Using Propositions~\ref{prop:countperfect} and~\ref{prop:countdmfs}, we can count the total number of dMfs for a simple infinite family of knot projections, obtained from the minimal diagrams $D_{2n+1}$ for the alternating torus knots $T_{2,2n+1}$, as in Figure \ref{fig:alternatingexample}.

Since the knot determinant of $T_{2,2n+1}$ is $2n+1$ \cite{rolfsen2003knots}, by Proposition~\ref{prop:perfectdMf} there are $2(2n+1)^2$ perfect discrete Morse functions on $\Xi(D_{2n+1})$.

We can use the symbolic Laplacian as in Proposition~\ref{prop:countdmfs} to generalise from perfect to general discrete Morse functions. The symbolic Laplacian matrices for the white and black graphs are respectively 
$$L^{symb}_w = \begin{pmatrix}
E & -E\\
-E & E
\end{pmatrix}$$
where $E = \sum_{i = 1}^{2n+1} e_i$, and
$$L^{symb}_b = \begin{pmatrix}
e_1^* + e_{2n+1}^* & -e_1^* & 0 & \ldots & 0 & -e_{2n+1}^*\\
-e_1^* & e_1^* + e_{2}^* & -e_2^* & \ldots & 0 & 0\\
0 & -e_2^* & e_2^* + e_{3}^* & \ldots & 0 & 0\\
\vdots & 0 &  \ddots & \ddots & 0 & -e_{2n}^* \\
-e_{2n+1}^* & 0 & \ldots & 0 & -e_{2n}^* & e_{2n}^* + e_{2n+1}^* 
\end{pmatrix}$$
The first symbolic characteristic polynomial is easily determined to be $p^{symb}_{G_w} = t(t -2E)$. So,
\begin{align}
    \#\mathcal{M}(D_{2n+1}) &= \restr{\left( (1 + 2E)\cdot p^{symb}_{G_b}(-1, e_1^*,\ldots, e_{2n+1}^*) \right)}{e_i = 1, e_i^* = 1}\nonumber \\
    &= p^{symb}_{G_b}(-1, 1,\ldots, 1) + \restr{2 \left(\sum_{i = 1}^{2n+1} e_i\cdot p^{symb}_{G_b}(-1, e_1^*,\ldots, e_{2n+1}^*) \right) }{e_i = 1, e_i^* = 1}\label{eqn:countegaim}.
\end{align}
Recall that all the products in Equation~\eqref{eqn:countegaim} are computed in the quotient polynomial ring $\faktor{\Z\left[ E(G_b),E(G_b^*)\right]}{\langle e_i \cdot e_i^*\rangle_{i = 1, \ldots , |E(G_b)|}}$.

Using Theorems $1$ and $2$ from \cite{chebotarev2008spanning} (which provide expressions for the number of rooted spanning forests in path and cycle graphs in terms of Fibonacci numbers) we can see that $$p^{symb}_{G_b}(-1, 1,\ldots, 1) = \Phi_{4n+1} + \Phi_{4n+3} - 2$$ and $$p^{symb}_{G_b}(-1, 1,\ldots,0,1,\ldots 1) =  \Phi_{4n+2},$$ where $\Phi_i$ is the $i$-th Fibonacci number. The second expression gives us the first term in Equation \eqref{eqn:countegaim}. To evaluate the second term in Equation \eqref{eqn:countegaim}, first recall that we are counting orthogonal forests in $G_b$ and $G_w$. Hence, for each $e_i$ in the sum $\sum_{i = 1}^{2n+1} e_i$, we can consider the path graph obtained from $G_b$ by deleting the edge $e_i^*$. Thus, we can use the second expression above, which counts rooted forests in path graphs. That is, the latter evaluation coincides with the evaluation in $(1, \ldots, 1)$ of $p^{symb}_{G_b}(-1, e_1^*, \ldots, e_{2n+1}^*) \cdot e_i$ in the quotient ring of Equation~\eqref{eqn:counting} for some edge $e_i \subset E(G_w)$ (as the polynomial is symmetric, we get the same result for any edge $e_i$). Putting all together we obtain this aesthetically pleasing result:
\begin{equation*}
\#\mathcal{M}(D_{2n+1}) = \Phi_{4n+1} + \Phi_{4n+3} + (4n+2)\Phi_{4n+2} - 2 = 2\Phi_{4n+1} + (4n+3)\Phi_{4n+2} -2.
\end{equation*}
\end{example}

\section{The Click-Clock theorem}\label{sec:clickclock}

In this section we generalise Kauffman's Clock Theorem to perfect admissible matchings, by introducing two new moves between partial Kauffman states. We also characterise the image of the KPW correspondence in terms of dMfs in Proposition~\ref{prop:clickpathmove}.

From this point onwards we need an extra assumption on the projection of the knot diagrams considered, namely that they are \emph{reduced}. This just means that the situation in Figure~\ref{fig:reduced} is not allowed; in knot theory, crossings like those shown in Figure~\ref{fig:reduced} are referred to as \emph{nugatory crossings}. This implies at once that both the black and white graphs are $2$-connected, meaning that they can not be disconnected by removing a single vertex.

\begin{figure}[ht]
  \centering
\includegraphics[width=6cm]{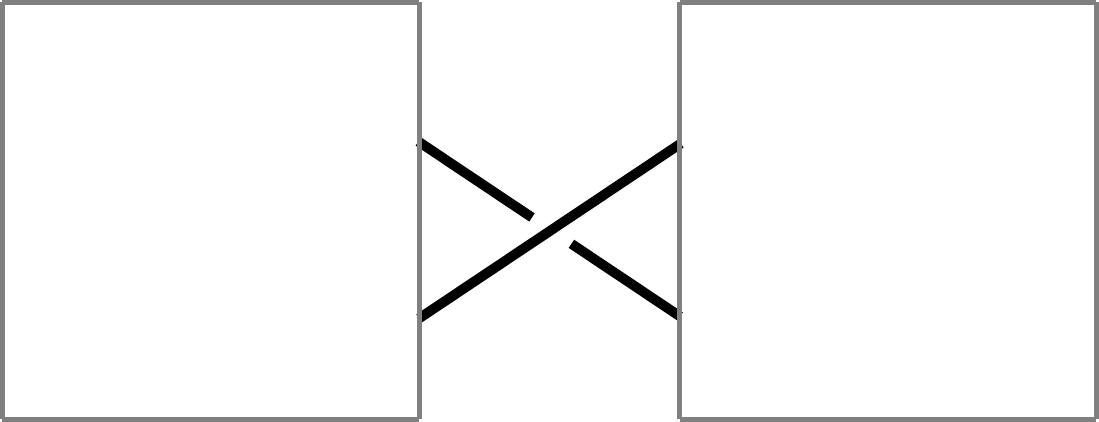}
\caption{A knot diagram is reduced if it does not contain any such configuration; the two boxes represent portions of the diagram.}
\label{fig:reduced}
\end{figure}

\begin{definition}\label{def:}
Consider two partial Kauffman states $x, x^\prime \in \widetilde{X}(D)$. A \emph{click loop move} from $x$ to $x'$ consists of altering $x$ on a single monochromatic loop on which $x$ and $x'$ disagree such that $x'$ and the modified $x$ agree on the loop. We say that $x^\prime$ and $x$ are \emph{click loop equivalent} if they both support the same monochromatic loops, and they coincide everywhere except on at least one of these loops (where they induce opposite orientations). 
\end{definition}

So, two click loop equivalent pKs are related by a finite number of click loop moves. 

\begin{definition}
Consider two maximal partial Kauffman states $x \in \widetilde{X} (D,v_w,v_b)$ and $x^\prime \in \widetilde{X} (D,v_w,v^\prime_b)$; we say that $x$ and $x^\prime$ differ by a \emph{click path move} $\rho \subset E(G_b)$ if $x$ and $x^\prime$ induce the same unrooted tree on $G_b$, $\rho$ is the unique black path in the branch of the tree determined by connecting the root $v^\prime_b$ to $v_b$, and the two partial states coincide everywhere except on $\rho$ (see Figure~\ref{fig:clickpath}). The same can be done if the two critical points differ only in the white graph. 
\end{definition}

\begin{figure}[ht]
  \centering
\includegraphics[width=7cm]{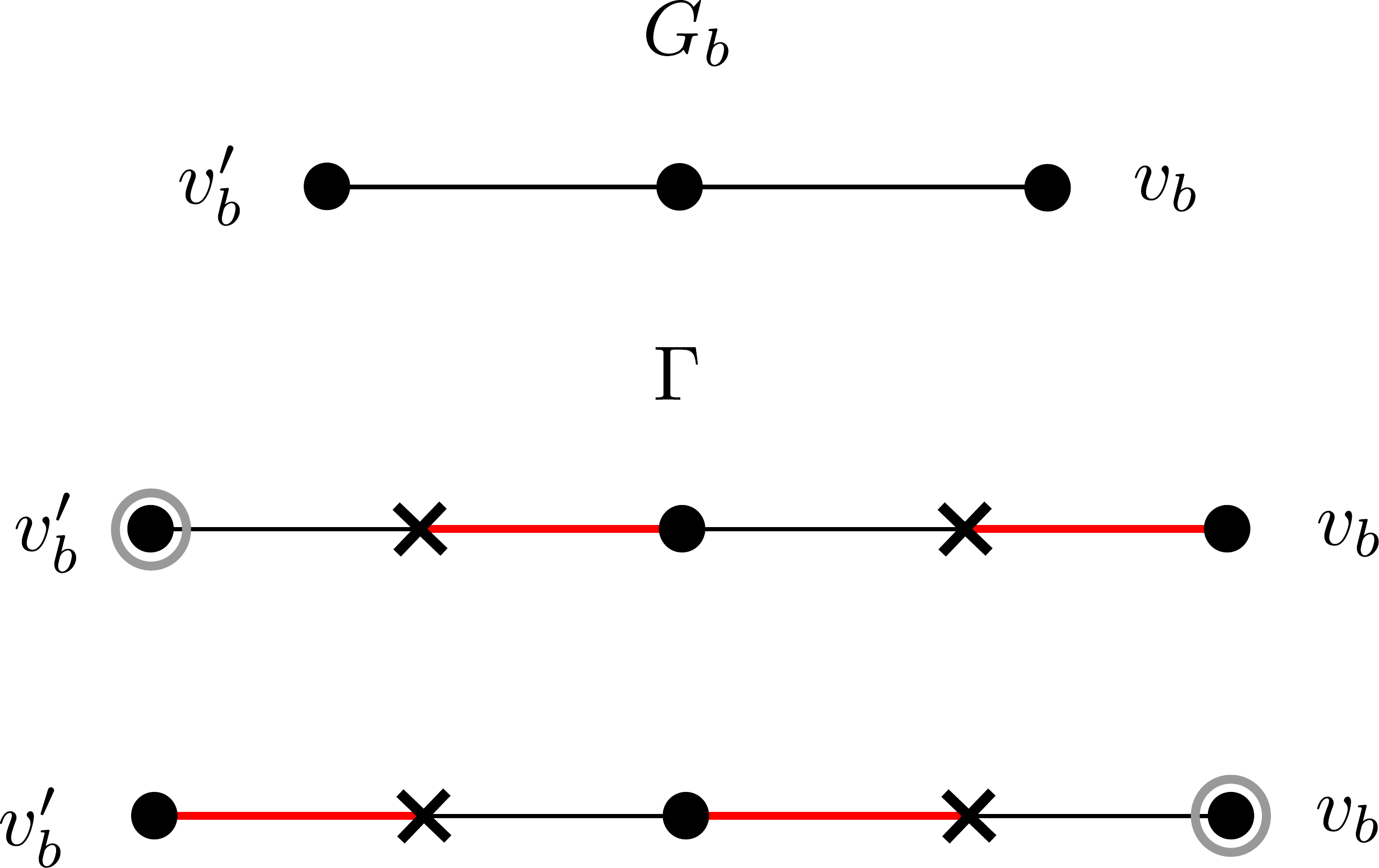}
\caption{At the top, a path $\rho$ of length 2 connecting the two roots in $G_b$, and on the bottom, the click path move associated with $\rho$ acting on the the two pKs. The two respective roots are circled in grey.}
\label{fig:clickpath}
\end{figure}

The result below characterises the image of the KPW map as the space of perfect acyclic matchings. In other words, the injection guaranteed by KPW's correspondence becomes a bijection if we restrict the codomain to perfect dMfs.

\sloppy
\begin{proposition}\label{prop:clickpathmove}
Click path moves induce bijections between perfect dMfs on $\Gamma(D,v_w,v_b)$ and perfect dMfs on $\Gamma(D,v_w^\prime,v_b^\prime)$ for any pair of black\slash white vertices. Moreover, at most two such moves are always sufficient.
\end{proposition}
\fussy 
\begin{proof}
First of all, let us note that a click path move does not change the maximality of the matching. It then suffices to prove the existence of a bijection for $X(D,v_w,v_b)$ and $X(D,a)$, where $a$ is an arc in $D$ separating the two regions $v_w$ and $v_b^\prime$. Fix $x \in X(D,v_w,v_b)$, and denote by $T_x$ the black spanning tree it supports. Call $\rho$ the unique path in $T_x$ connecting the root $v_b$ to $v_b^\prime$. A click path move along $\rho$ transforms $x$ into a matching $x^\prime \in X(D,a)$ (again supporting $T_x$ as a spanning tree). We can then conclude by observing that a click path move does not introduce or remove any monochromatic loops supported by either $x$ or $x^\prime$ (see Figure \ref{fig:clickpathnoloop}); the only way this would be possible is if a supported loop could be adjacent to a path along which we can apply a click path move. However, we cannot obtain such a configuration without introducing an edge in the matching incident to two crossings.
Hence, these moves preserve the acyclicity of maximal pKs by Theorem~\ref{thm:dmfandstates}.

\begin{figure}[ht]
  \centering
\includegraphics[width=9cm]{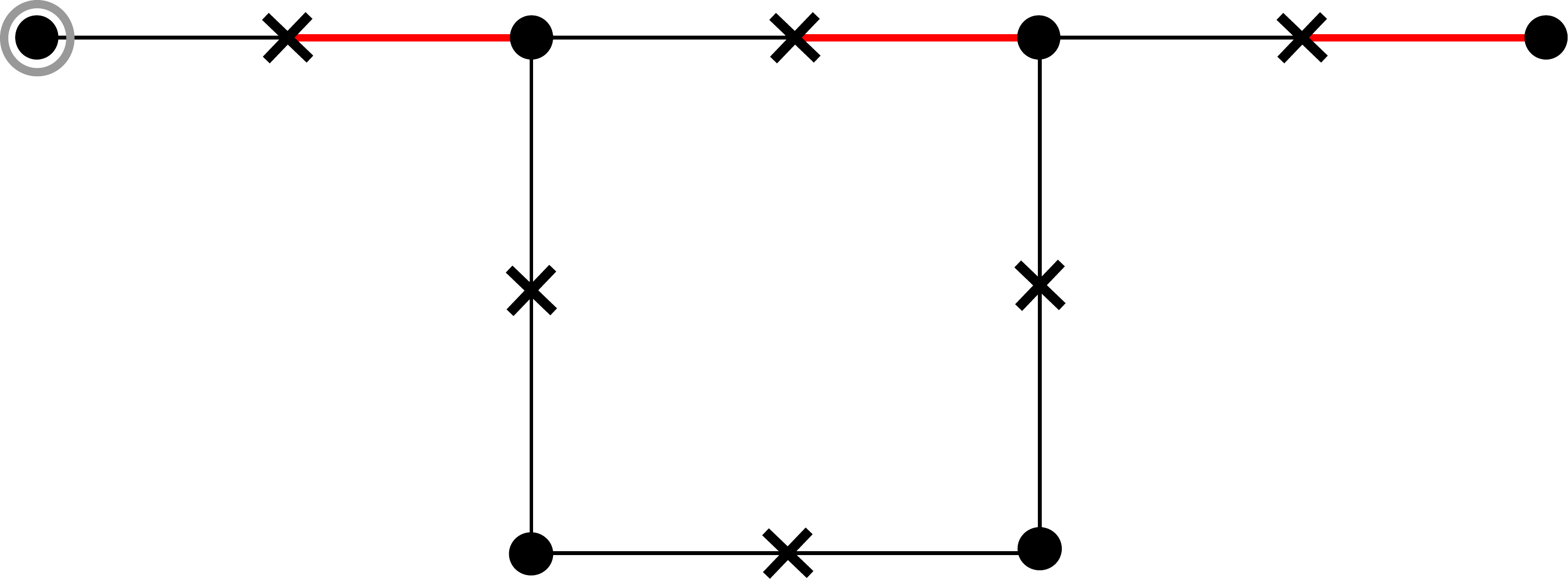}
\caption{A click path move cannot introduce or remove any monochromatic loops supported by a perfect dMf; an example of the contradiction in the proof of Proposition~\ref{prop:clickpathmove}.}
\label{fig:clickpathnoloop}
\end{figure}

The general statement then follows easily by noting that a single click path move along a black\slash white path suffices to change the two unmatched vertices of any perfect dMf on $\Xi(D)$.
\end{proof}

\begin{corollary}\label{cor:alldmfsfromKPW}
There is a bijection between the union of all elements in the image of the KPW correspondence over all pairs $(v_b,v_w) \in G_b \times G_w$ and perfect dMfs on $\Xi(D)$.
\end{corollary}
\begin{proof}
Each perfect matching on $\Gamma(D)$ which is in the image of the KPW
correspondence is uniquely determined by the triple $(T,v_b,v_w)$, with $T \in Tree(G_b)$. By Theorem \ref{thm:dmfforest}, each such triple is in bijective correspondence with a perfect discrete Morse function on $\Xi(D)$.
\end{proof}

Neither click loop moves nor click path moves change the Jordan trails; clock moves and click loop moves do not change the position of the critical points. In fact, when perfect discrete Morse functions differ only by click path moves, we have the following result (\emph{cf.} also Lemma~\ref{lem:samejordan}).

\begin{lemma}\label{lem:sametrail}
Two perfect dMfs on $\Xi(D)$ share the same Jordan trail if and only if they support the same unrooted spanning trees.
\end{lemma}
\begin{proof}
\begin{figure}[ht]
  \centering
\includegraphics[width=6cm]{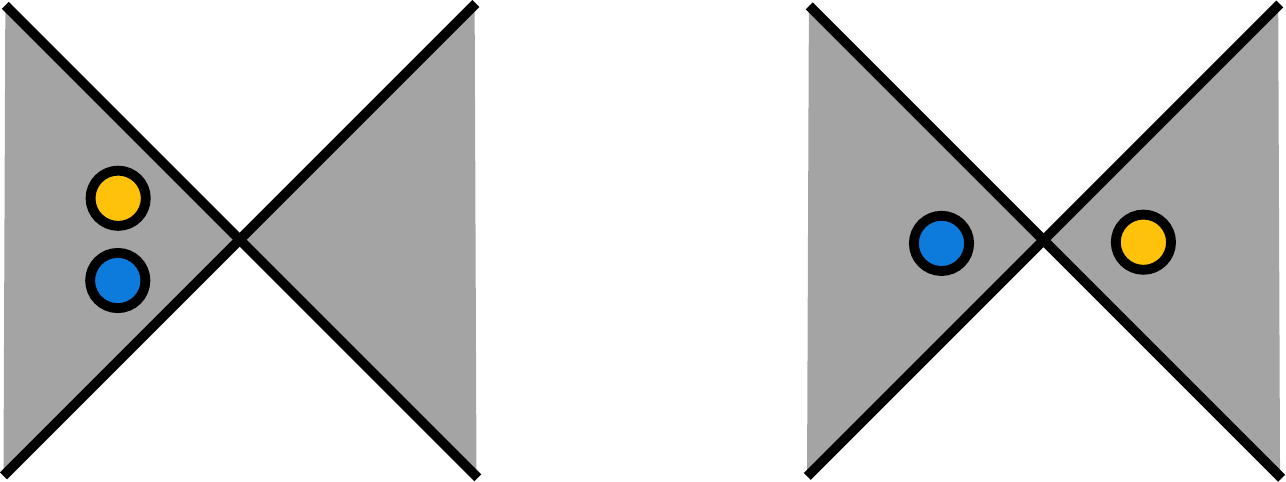}
\caption{The two possibilities at each crossing: the two matchings either coincide or are on opposite sides of the crossing.}
\label{fig:2caseslemma}
\end{figure}
Let $x$ and $x^\prime$ be two perfect dMfs such that $J(x) = J(x^\prime)$. Then at each crossing we must have one of the two possibilities shown in Figure~\ref{fig:2caseslemma}, where the dots are, without loss of generality, in the black regions; in both cases we immediately see that the black edge between the two regions in figure must belong to the subgraph of the black graph supported by the dMfs; by Theorem~\ref{thm:dmfandstates} this subgraph can not contain cycles, thus it is a tree. 

Conversely, if an edge of the black graph belongs to the common spanning tree, than on the corresponding crossing we are in the situation of Figure~\ref{fig:2caseslemma}, and hence the two dMfs have the same Jordan trail.
\end{proof}

\begin{lemma}\label{lem:changeby2}
Given a perfect admissible matching $x$ on $\Gamma(D)$, a clock move can either leave $|J(x)|$ unchanged or change it by $\pm 2$.
\end{lemma}
\begin{proof}
This follows easily by trying out all possible cases, shown in  Figure~\ref{fig:casesclock}, which shows (up to symmetries) how clock moves change the Jordan resolution. The number in the middle indicates how $|J(x)|$ changes under the clock move, while the dashed line indicates how the local picture eventually joins up. We divide the three possible types depending on their action on the resolution. Type I is the only one needed in the Clock theorem, Type II can only occur for matchings that are not dMfs (as necessarily $|J(x)|>1$), and Type III moves are the only ones that can merge\slash split Jordan cycles.
\begin{figure}[ht]
  \centering
\includegraphics[width=9cm]{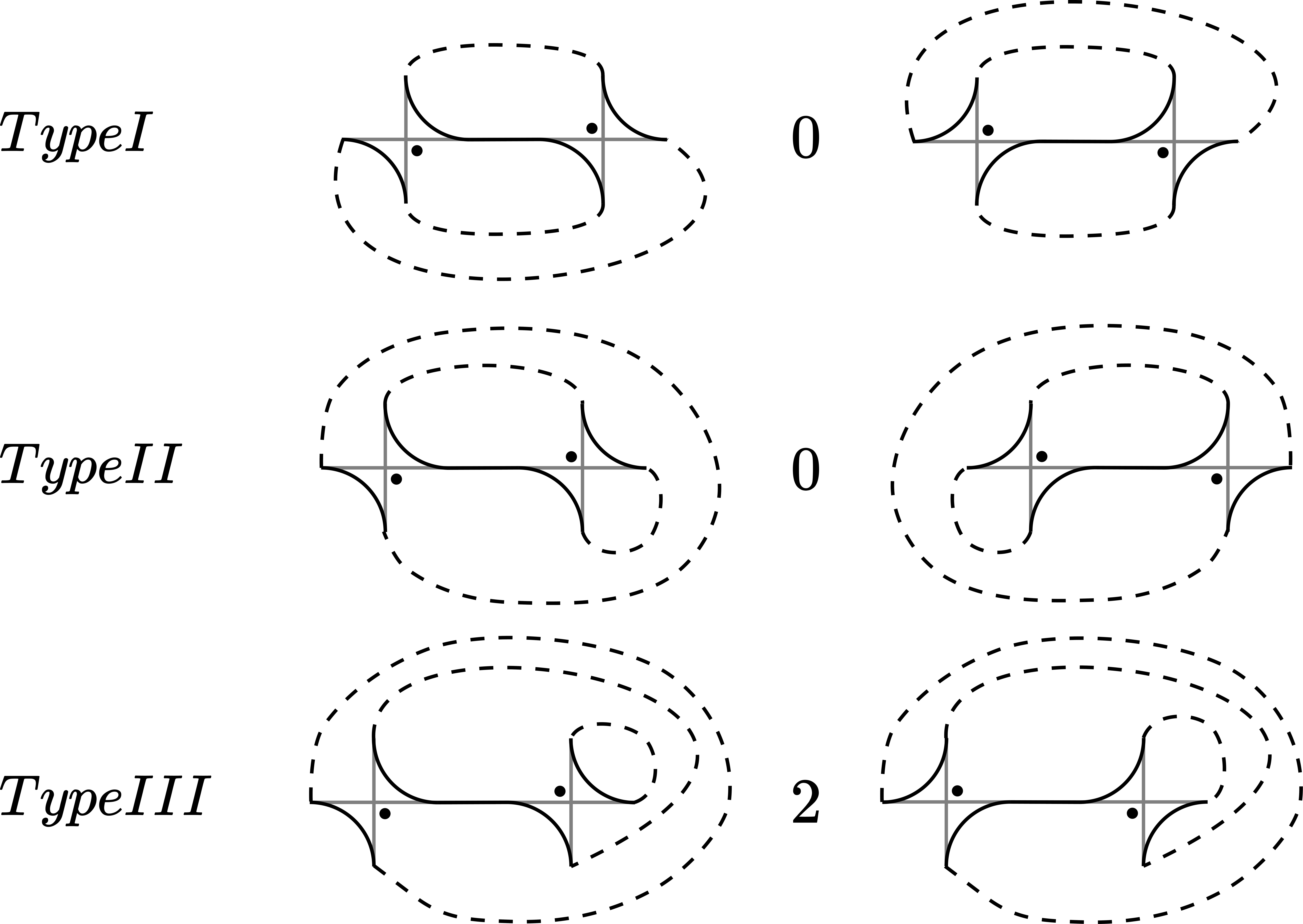}
\caption{All cases (up to symmetries) of how clock moves change the Jordan resolution.}
\label{fig:casesclock}
\end{figure}
\end{proof}

\begin{proposition}\label{prop:whichcomponents}
If $x \in \widetilde{X}(D)$ is a perfect and admissible matching, then each connected component of the black and white orthogonal subgraphs $H_x^b \subset G_b$ and $H_x^w \subset G_w$ it induces is either a tree or has the homotopy type of a circle. Furthermore, there is exactly one tree of each colour.
\end{proposition}
\begin{proof}
Call $H$ such a connected component, and assume without loss of generality that $H$ is a subgraph of the black graph. Then $H \sim \bigvee^m S^1$ for some $m \ge 0$. Now, the Euler characteristic tells us that the difference between the number of black vertices and crossings is $1-m$. 

\begin{figure}[ht]
  \centering
\includegraphics[width=11cm]{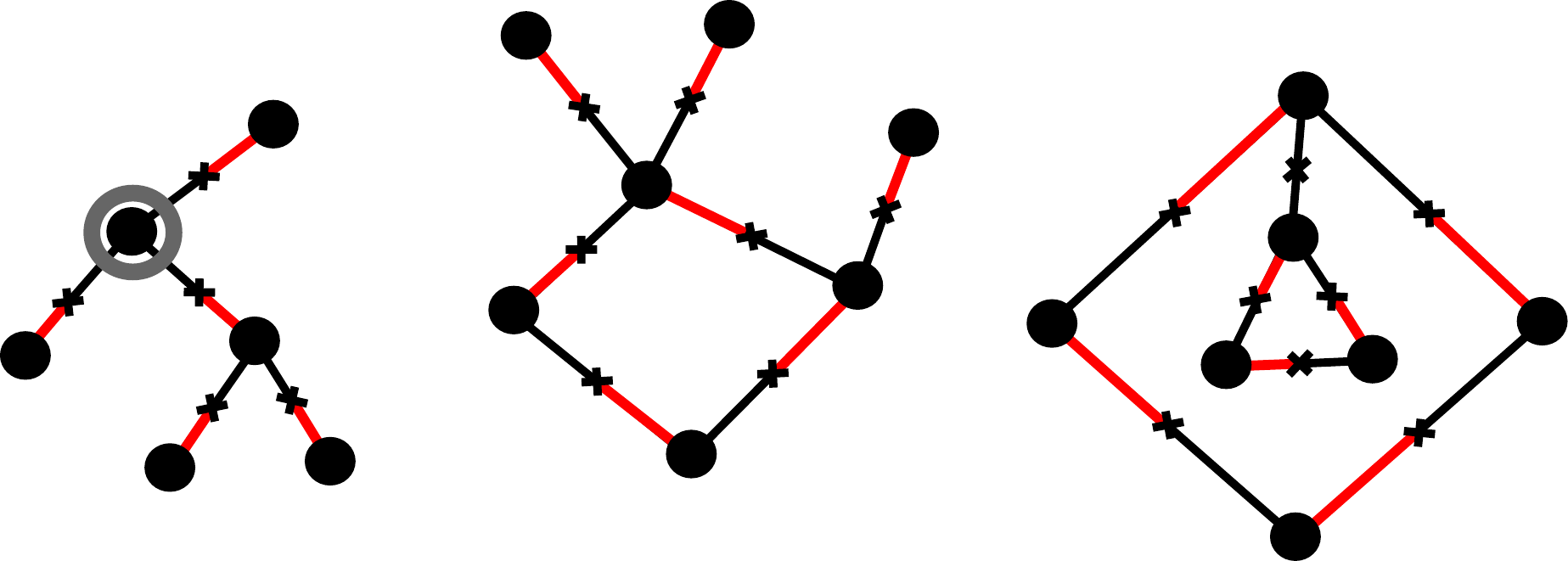}
\caption{On the left, a tree with its circled root--which is the unpaired vertex; half edges are marked with a cross, and matching between crossings and vertices are highlighted in red. In the middle a component with the homotopy-type of a circle. Here there are no unmatched vertices, which must instead be in the complement in $S^2$ of this subgraph (\emph{cf}.~Prop.~\ref{prop:perfectdMf}). The configuration on the right (with rank $\ge 2$) cannot have all of its vertices matched.}
\label{fig:ranksofsubgraphs}
\end{figure}

Hence we can distinguish $3$ cases, shown in Figure \ref{fig:ranksofsubgraphs}; if $m = 0$, then $H$ is a tree, and there is always a perfect matching on (the first barycentric subdivision of) $H$ leaving out exactly one black vertex. If $m=1$, then $H$ is homotopically a circle, so it is a simple cycle with some trees ``attached'' to it. In this case the number of black vertices coincides with the number of crossings, and there is (using the assumption that the matching is perfect) a perfect matching on $H$ (in fact it is possible to prove that there are exactly two) which necessarily contains a monochromatic loop. On the other hand, if $m>1$ there are more crossings than black dots, so there is no hope for a portion of a perfect admissible\footnote{As the only unmatched vertices can only be black or white, not crossings.} matching to be supported by $H$. As trees are the only components that can support an unmatched vertex, the admissibility of the matching implies that there is exactly one of each colour.
\end{proof}

The following result characterises admissibility in terms of the parity of the Jordan resolution and is used in the proof of Theorem \ref{thm:clickclock}; as a consequence of the proof, we will also see that the subgraphs induced by admissible and perfect matchings are concentric.

\begin{proposition}\label{prop:admissibleodd}
A perfect matching $x \in \widetilde{X}(D)$ is admissible if and only if $|J(x)|$ is odd.
\end{proposition}
\begin{proof}
It follows from the proof of Lemma~\ref{lem:unpaired} that each monochromatic loop supported by $x$ has exactly one unpaired vertex in both discs into which it divides $S^2$. Furthermore, by Proposition~\ref{prop:whichcomponents} each connected component of the black and white subgraphs $H_x^b$ and $H_x^w$ can contain at most one monochromatic loop.

Then, by Lemma~\ref{lem:unpaired}, each connected component of $J(x)$ must have an unmatched vertex on both sides; here we are using the fact that monochromatic loops create connected components in the Jordan resolution. Since $x$ is perfect, there are only two unmatched vertices in $\Gamma(D)$; this forces all Jordan cycles in $J(x)$ to be concentric on the $2$-sphere (see Figure~\ref{fig:circlesofhell}).  
\begin{figure}[ht]
  \centering
\includegraphics[width=9cm]{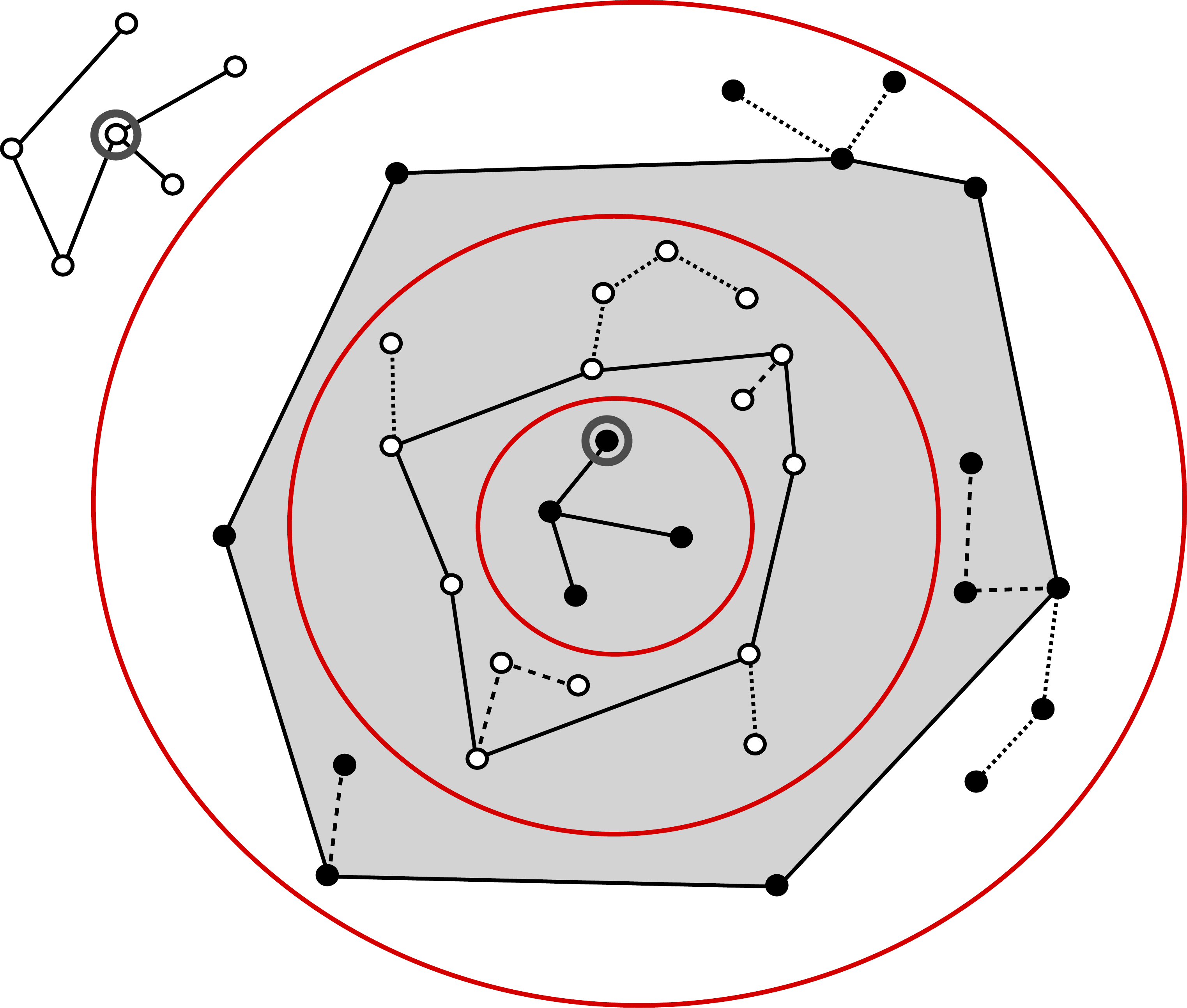}
\caption{A schematic view of a generic admissible matching up to isotopy on $S^2$; the white and black graphs here are the subgraphs of $G_w$ and $G_b$, respectively, induced by a maximal pKs. In the central disc, there is the isolated black rooted tree (roots are circled), while the white one is on the outside. The other components are the pseudotrees, and external trees are dotted while internal ones dashed.
The only unmatched vertices are the roots of the two trees. The red circles represent the $3$ Jordan cycles. The grey shaded region contains the innermost black cycle, its internal trees and the unique isolated black tree.}
\label{fig:circlesofhell}
\end{figure}

Connect the two unpaired vertices with a path in $S^2$ that avoids the crossings of $D$. Each time the path crosses an arc in the projection it changes colour. Then, note that the parity of the number of the intersections coincides with the parity of the number of circles in any possible Jordan resolution having those two vertices as the unmatched ones.
\end{proof}

As a consequence of this last result, we see that  for a given perfect admissible matching $x$, $|J(x)| - 1$ coincides with the number of monochromatic loops in $x$.

It follows from the last two results that the subgraph $H_x^b$ of the black graph  induced by a perfect admissible matching is a special case of a spanning \emph{pseudoforest}. By this we mean that each connected component is either a tree or has the homotopy type of a circle (it is a \emph{pseudotree}). Up to isotopy of $S^2$, the unique black tree can be thought of as being in the centre of all the concentric pseudotrees. We call this tree \emph{isolated}; all other components of the black pseudoforest can be decomposed into three parts: a \emph{cycle} (\emph{i.e.}~a single simple closed loop in $G_b$ induced by a monochromatic loop), the \emph{internal} trees and the \emph{external} trees, each of which is attached to the cycle at exactly one vertex.
We can distinguish between internal and external trees, according to whether they are respectively in the disc bounded by the cycle that contains the isolated tree of the same colour or not. The schematic picture is summarised in Figure~\ref{fig:circlesofhell}.
\bigskip

We recall here the statement of the simultaneous generalisation of the Clock Theorem by Kauffman~\cite[Thm.~2.5]{kauffman2006formal} to perfect matchings and of the main result in~\cite[Thm.~1]{kenyon1999trees}. In particular, Theorem~\ref{thm:clickclock} gives a way to organise all possible discrete Morse functions on a given cellular structure on $S^2$ in the image of $\Xi$, where each pair of discrete Morse functions is related by a finite sequence of moves.

\begin{theorem}[Click-Clock]\label{thm:clickclock}
If the diagram $D$ is reduced, any two perfect admissible matchings on $\Gamma(D)$ are related by a finite sequence of click path, click loop and clock moves.
\end{theorem}

As an easy consequence of the proof, two perfect dMfs can be transformed into one another by clock and click path moves -- or only clock moves of Type I if they share the same critical points (this last part is Kauffman's original Clock Theorem); furthermore, if these two dMfs share the same Jordan trail, then they differ by at most two click path moves.

Theorem \ref{thm:clickclock} will follow from the next two lemmas; the proof of the first one is a straightforward generalisation of Lemma \ref{lem:sametrail} to perfect admissible matchings. The proof of the second one is instead rather involved, and will occupy most of the remaining section.

\begin{lemma}\label{lem:samejordan}
Two perfect admissible matchings on $\Xi(D)$ have the same Jordan resolution if and only if they are related by a finite sequence of click path and loop moves.
\end{lemma}
\begin{proof}
The ``if'' direction is trivial, as both kinds of moves preserve the Jordan resolution.

For the other implication, let $x$ and $x^\prime$ denote two perfect admissible matchings satisfying $J(x) = J(x^\prime)$. Observe that $J(x) = J(x^\prime)$ if and only if all of the crossings in the corresponding diagrams resolve in the same way. This gives two possibilities, illustrated in Figure~\ref{fig:2caseslemma}; the two matchings can either coincide at a crossing or be on opposite sides (on the two incident regions with the same colour). 

If $x=x^\prime$ then the statement follows immediately. If not, then there must be at least one crossing $c$ on which the second possibility occurs. Without loss of generality, suppose we start at such a crossing, where the components of the two states are in the black regions. Draw an edge in $G_b$ if the crossing through which the edge passes is as in the right of  Figure~\ref{fig:2caseslemma}. Starting from $c$, draw as many such edges as possible in both directions. If this path forms a loop then the two states are related by a click loop move. If not, then there is a (unique) path between the two unmatched black regions, one from $x$ and one from $x^\prime$, in which case the two states are related by a click path move. To see that this is actually a path, observe that by construction we never have a vertex of degree three or more, because this would imply the existence of at least two components in a single region from the same matching, which violates the definition of a partial Kauffman state. Since the matchings are admissible, they each have exactly one unmatched black region, hence this path must be between the two unmatched regions in the different matchings with the same colour.
\end{proof}

\begin{lemma}\label{lem:connectedjordan}
Up to clock moves and click loop and path moves, the Jordan trail of a perfect admissible matching can be made connected.
\end{lemma}

Since the proof of this last lemma is going to take a bit of work, we first show that, together with Lemma \ref{lem:samejordan}, it does in fact imply the Click-Clock theorem:

\begin{proof}[Proof of Theorem~\ref{thm:clickclock}]
It suffices to show that for a given perfect and admissible matching $x$ on $\Gamma(D)$, we can convert it to a perfect dMf of $\Gamma(D)$ with prescribed adjacent critical points; the thesis would then follow by applying the ``classical'' Clock Theorem.

Using Lemma~\ref{lem:connectedjordan} we can apply a finite sequence of clock and click loop moves on $x$ that make its Jordan trails connected. Call this new matching $\widetilde{x}$; by Corollary~\ref{cor:dmfconnectedJ} we know that $\widetilde{x}$ is in fact a dMf on $\Gamma(D)$, and up to click path moves we can transform it into yet another dMf with prescribed and adjacent critical points, by Proposition~\ref{prop:clickpathmove}.
\end{proof} 

Before we start the proof of Lemma \ref{lem:connectedjordan}, we need to introduce several new objects and study some of their properties. In particular we need the following \emph{leaf spin} operation, first considered in \cite{harary1983interpolation}. A \emph{leaf} is an edge with at least one incident vertex of degree one.

\begin{definition}
Let $G$ be a plane graph, and $H \subset G$ a subgraph. If $\ell$ is a leaf in $H$ which is not a leaf in $G$, define the clockwise (resp.~counterclockwise) \emph{spin} of $H$ along $\ell$ as the subgraph $H^\prime$ of $G$ obtained by removing $\ell$ from $H$, and adding the first edge encountered by spinning $\ell$ around its isolated endpoint in a clockwise (resp. counterclockwise) fashion (see Figure~\ref{fig:SPIN}).
\end{definition} 
\begin{figure}[ht]
  \centering
\includegraphics[width=9cm]{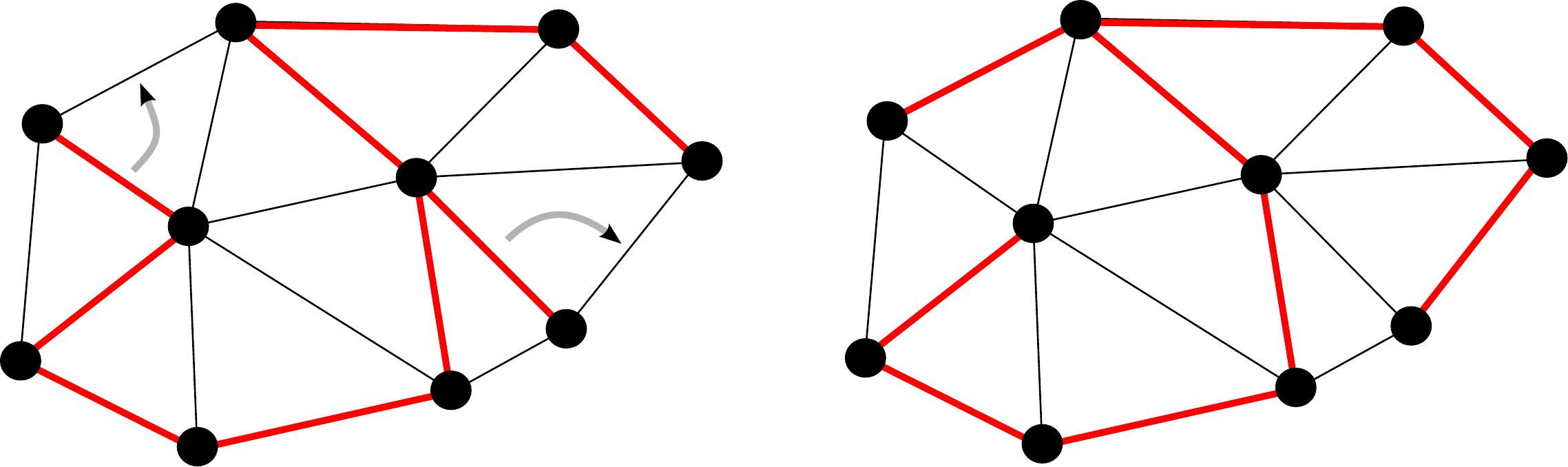}
\caption{On the right a plane graph with a spanning tree (highlighted in red); on the right the effect of two leaf spins on the tree, one clockwise and one counterclockwise. }
\label{fig:SPIN}
\end{figure}

If $H$ is a spanning forest (or pseudoforest) then leaf spins preserve this structure as pivoting on a leaf does not change the fact that it is a leaf. Hence, the graph $H^\prime$ obtained from $H$ by a leaf spin is also a spanning forest (or pseudoforest). Furthermore, a leaf spin always leaves the cycles of a pseudoforest unchanged. We need to take a bit more care when dealing with rooted pseudoforests, as this will be the case we need. If a component of the pseudoforest contains the root as the terminal vertex of the leaf we want to spin, we first need to move the root along the unique edge to which it is incident. In other words we never spin a leaf around the root; this way the root does not switch between connected components. As a special case, if the rooted connected component consists of only one edge, spinning it will leave the root as an isolated vertex.\bigskip

We are about to see how leaf spinning will yield the connection between concentric pseudoforests induced by admissible perfect matchings and certain clock moves on them, ultimately allowing us to conclude the proof of Lemma~\ref{lem:connectedjordan} and thus Theorem \ref{thm:clickclock}. More precisely, consider a pair of spanning pseudoforests induced by admissible perfect matchings, coinciding everywhere outside the first black cycle. We will prove that these are related by leaf spins (and hence by clock and click path moves).

\begin{lemma}\label{lem:spinandclock}
If the spanning concentric pseudoforest $H_x^b$ induced by the perfect admissible matching $x \in \widetilde{X}(D)$ contains a leaf, then all of the pseudoforests obtained by spinning this leaf are induced by matchings that are clock and click path equivalent to $x$.
\end{lemma}
\begin{proof}
Call $R \in V(G_b)$ the region corresponding to the terminal black vertex. Let us start by noting that a leaf in $H_x^b$ is necessarily ``surrounded'' by white edges, as shown in Figure~\ref{fig:spinmatching}. Consider the two cases shown in Figure \ref{fig:spinmatching}: in the left one we can spin the terminal black leaf in a clockwise fashion, and this corresponds to a sequence of clock moves. After one spin, this local configuration is just the one we started with, only tilted by the angle of the spin. Hence, we can reach all other possible edges dual to the white ones through clock moves. In the second case we can start clocking\slash spinning the leaf in both directions, and also reach all other edges in $G_b$ sharing the terminal vertex with the leaf.

\begin{figure}[ht]
  \centering
\includegraphics[width=11cm]{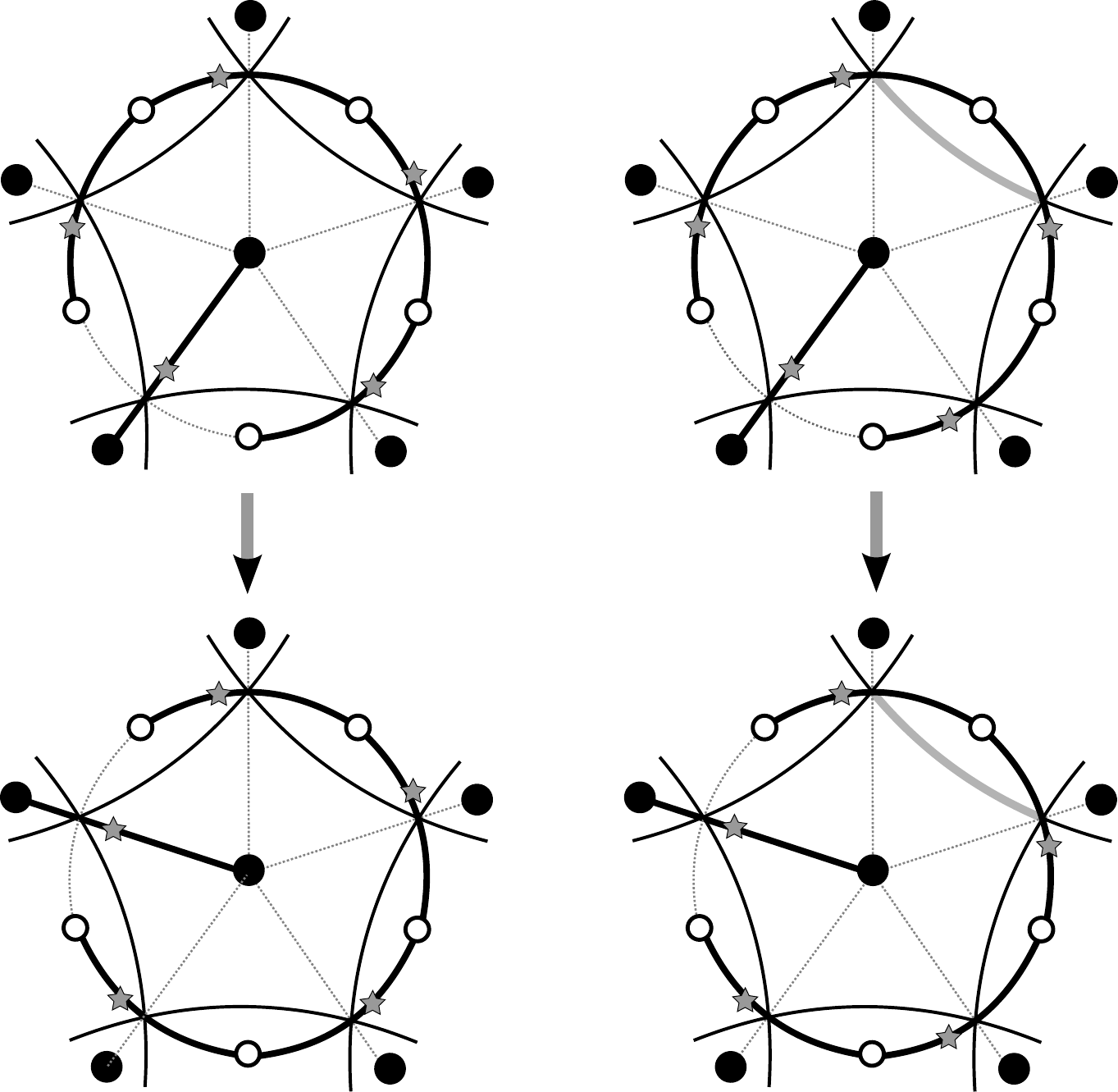}
\caption{Bold black edges are those included in the subgraphs induced by the matching, and stars mark the dots of the associated pKs. In the top-left configuration, we can only spin the black leaf in a clockwise fashion, and this can be achieved through a sequence of clock moves, the result of which is shown in the bottom-left of the figure. In the top-right configuration we can spin and apply clock moves in both directions; note that we cannot clock clockwise (resp. counterclockwise) after we reach the two edges at the sides of the unique white region (the bold grey arc in the top-right) not matched with a crossing adjacent to the region $R$. In the bottom-right configuration we can see the result of applying a clockwise leaf spin to the top-right configuration.}
\label{fig:spinmatching}
\end{figure}

It remains to show that these are all of the possible configurations containing leaves; \emph{a priori} there are two possible things that could go wrong. Firstly, there could be no component of the pKs in nearby white regions that allows for a clock move in either direction (and hence the spin could not correspond to a clock move); this can be excluded by looking at the left part of Figure \ref{fig:spinnotclock}.

Lastly, applying multiple leaf spins might result in a configuration where more than one edge of $R$  does not contain any component of the pKs on the crossings it has in common with $R$. This configuration contradicts the maximality of the pKs and the fact that there is a leaf, as illustrated in the right part of Figure~\ref{fig:spinnotclock}.

Note that if any of these leaf spin moves involve the root, then we need to apply a suitable click path move to shift it away.

\begin{figure}[ht]
  \centering
\includegraphics[width=11cm]{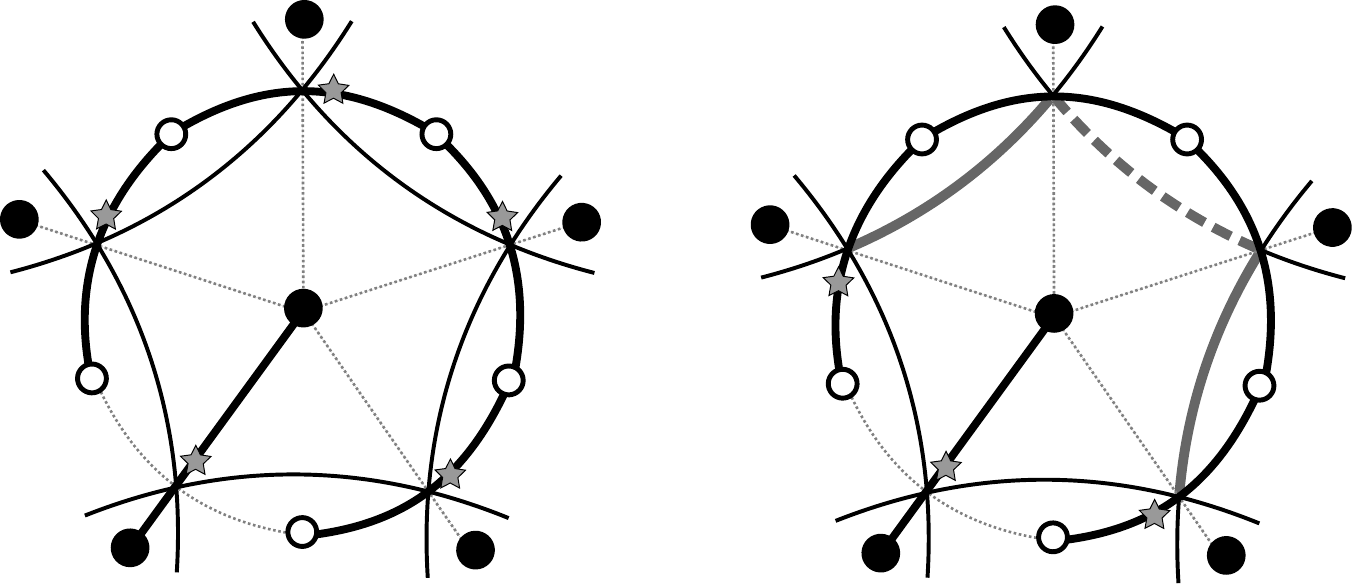}
\caption{The two impossible configurations from the proof of Lemma \ref{lem:spinandclock}. 
In the first case the stars near the only matched black edge force an impossible configuration: the two crossings in the top right can not be matched with any region, compatibly with the existence of the black leaf. In the second case, if we assume that the two bold grey arcs have no stars, there is no way of completing the partial Kauffman state in the grey bold and dashed edge (again, if we want to preserve the black leaf). }
\label{fig:spinnotclock}
\end{figure}
\end{proof}

\begin{remark}\label{rmk:whatspindoes}
It follows immediately that spinning a leaf always correspond to clock moves of Types I and II (see Figure \ref{fig:casesclock}); in particular, a leaf spin can never merge different components of $J(x)$. Furthermore, a leaf spin cannot alter the cycle in a pseudoforest, but might change the cycle of the \emph{dual} pseudoforest.
\end{remark}

To complete the proof of Lemma~\ref{lem:connectedjordan} by proving the existence of a sequence of moves decreasing the number of Jordan cycles, we must take a detour to introduce some technical tools.

\begin{definition}
An \emph{almost-tree} is the spanning forest obtained by removing exactly  one edge from a spanning tree (see Figure~\ref{fig:COMPACTIFICATION}). If the edge removed is a leaf, then its isolated vertex will be considered as a connected component. 
\end{definition}

\begin{proposition}\label{prop:connectedgraph}
If $G$ is a $2$-connected plane graph, denote by  $\widehat{G}_{aT}$ the graph whose vertices are almost-trees in $G$ and edges are leaf spins. Then  $\widehat{G}_{aT}$ is connected.
\end{proposition}
\begin{proof}
This is an easy consequence of \cite[Thm.~1]{harary1983interpolation}. The authors of \cite{harary1983interpolation} prove that the related graph $G_T$, whose vertices are spanning trees and edges are leaf spins, is connected. We can conclude by noting that in $\widehat{G}_{aT}$ there are more possible moves that can be performed, as any almost-tree can arise after the removal of an edge in several trees in $G$.
\end{proof}

Recall that Theorem \ref{thm:dmfforest} and Proposition \ref{prop:whichcomponents} imply that for an admissible perfect matching $x \in \widetilde{X}(D)$ we get two induced pseudoforest subgraphs $H_x^w \subset G_w$ and $H_x^b \subset G_b$ composed by concentric pseudotrees alternating in colour.

Consider the subgraph $\overline{G}_b$ of $G_b$ composed of the edges in the innermost black cycle --that is, the unique black pseudotree that surrounds the unique black tree-- and all of the vertices and edges within.

In other words, $\overline{G}_b$ is the subgraph of $G_b$ bounded by the unique black cycle, in the disc containing the internal black trees. Let $\overline{H}_x^b$ denote the intersection of $\overline{G}_b$ with the pseudoforest $H_x^b$. Note that $\overline{H}_x^b$ is spanning in $\overline{G}_b$, and has exactly two connected components: one is a pseudotree (containing only internal trees), and the other is the isolated black tree. 

A finite $2$-connected graph $G$ embedded in the plane, divides $\R^2$ in regions; we call the edges of G in the boundary of the external infinite region its \emph{boundary cycle}. With the assumptions above, this is in fact a simple cycle in $G$ whose removal from the plane splits it into two discs. 

\begin{definition}
For a given finite, leafless, $2$-connected and planar graph $G$ embedded in $\R^2$, call $C(G) \subset S^2$ its \emph{boundary contraction}, the graph obtained by collapsing the boundary cycle of $G$ to a point.
\end{definition}
\begin{figure}[ht]
  \centering
\includegraphics[width=10cm]{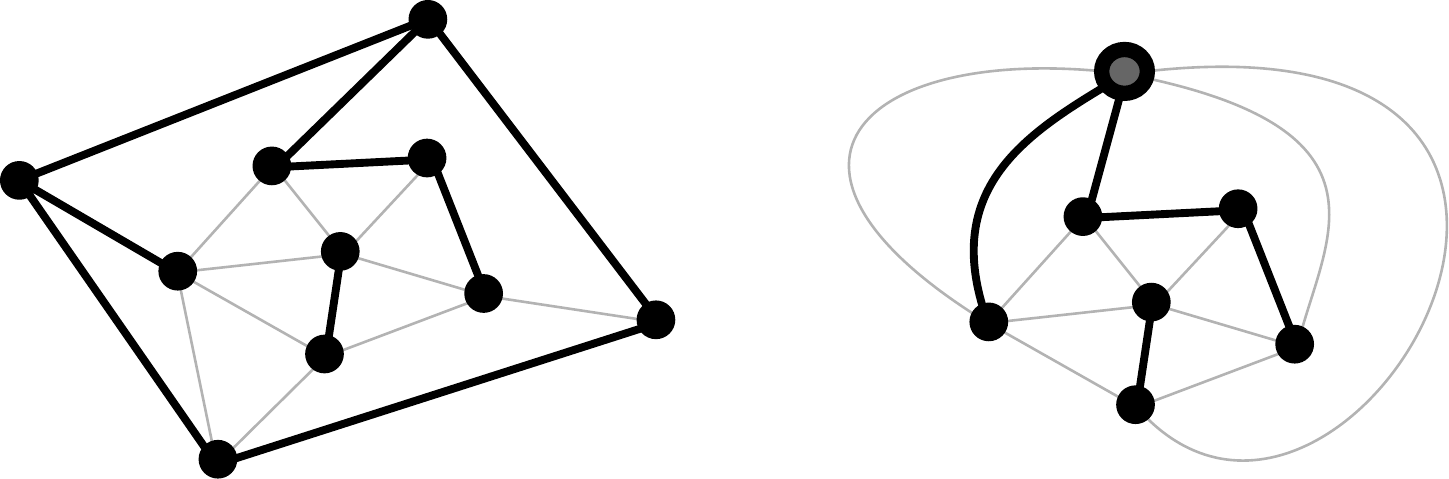}
\caption{On the left one example of $\overline{G}_b$ graph, with $\overline{H}_x^b$ in bold. On the right the collapsed boundary cycle  is the top grey vertex; we can also see how the isolated tree, together with the first black cycle and its internal trees, correspond to an almost-tree in the boundary contraction $C(\overline{G}_b)$.}
\label{fig:COMPACTIFICATION}
\end{figure}

By taking the boundary contraction  we can associate to $\overline{H}_x^b \subset \overline{G}_b$ a spanning almost-tree $C(\overline{H}_x^b)$ in $C(\overline{G}_b)$, with one connected component consisting of the image of all the internal trees, and the other one of the image of the isolated tree.

\begin{proof}[Proof of Lemma~\ref{lem:connectedjordan}]
Let us denote by $x$ a perfect admissible matching with $|J(x)| = 2k+1$, since the number of components in the Jordan resolution is necessarily odd by Proposition~\ref{prop:admissibleodd}. If $k = 0$ there is only one connected component, and we are done.

We want to show the existence of a sequence of clock moves and click path\slash loop moves that take $x$ to a new perfect admissible matching $x^\prime$ with $|J(x^\prime)| = 1$. We do this by induction on $k$. The strategy is to reduce to a configuration akin to the left side of the third case in Figure \ref{fig:casesclock}, so that we can apply one final clock move in order to reduce the number of Jordan cycles by $2$. The result then follows by the inductive step.

This can be achieved by suitably changing the subgraph $\overline{G}_b$ of the black graph, defined above. 

The idea is the following: if we can make the two innermost cycles (\emph{i.e.}~the black cycle and the white cycle contained in $H^w_x$ which separates the black isolated tree from the black cycle) as close as possible (see the left of Figure~\ref{fig:bothoncycle}), then --up to possibly a click loop move on either cycle-- we can apply a clock move that merges the innermost Jordan cycles. Applying a click loop move ensures that there is at least one square in $\Gamma(D)$ with opposite edges of a square being matched, which makes applying a clock move possible.

We start by modifying the isolated black tree, by adding to it all the edges belonging the internal trees of the first pseudoforest $\overline{H}_x^b$. We can do this using clock moves and possibly click path moves thanks to Proposition~\ref{prop:connectedgraph}. That is, we apply a sequence of leaf spins to move the internal trees of $\overline{H}_x^b$ to the isolated black tree. Leaf spins correspond to sequences of clock moves and changing the roots corresponds to click path moves.

Before we can conclude, we need to eliminate some local configurations that would prevent our strategy from succeeding; in particular, after spinning all internal black trees onto the unique isolated black tree as described in the previous paragraph, there remain configurations that prevent the application of a single clock move to reduce the number of Jordan cycles by 2, as illustrated in Figure~\ref{fig:casesclock}. After applying the procedure described in the previous paragraph, all external white trees are in fact paths. This is because we ``spun away'' all black internal trees in $\overline{H}_x^b$. A configuration that would prevent applying the desired final Type III clock move is the presence of an external white tree surrounded by the unique black cycle, as illustrated in the centre of Figure~\ref{fig:bothoncycle}. The problem here is that these external linear white trees separate the innermost white cycle from the innermost black cycle, preventing reaching the desired configuration (see the left part of Figure~\ref{fig:bothoncycle}) where we can apply the final clock move. In particular, this occurs if there exists an edge in $\overline{G}_b \setminus \overline{H}_x^b$ with both vertices on the innermost black cycle (see the grey dashed edge in the central part of Figure~\ref{fig:bothoncycle}). 

We can overcome this problem by applying a sequence of leaf spins on these external linear white trees, which can be seen in the central and right parts of Figure~\ref{fig:bothoncycle}. Note that these leaf spins have the effect of ``shrinking'' the innermost black cycle. 

\begin{figure}[ht]
  \centering
\centering
\includegraphics[width =12cm]{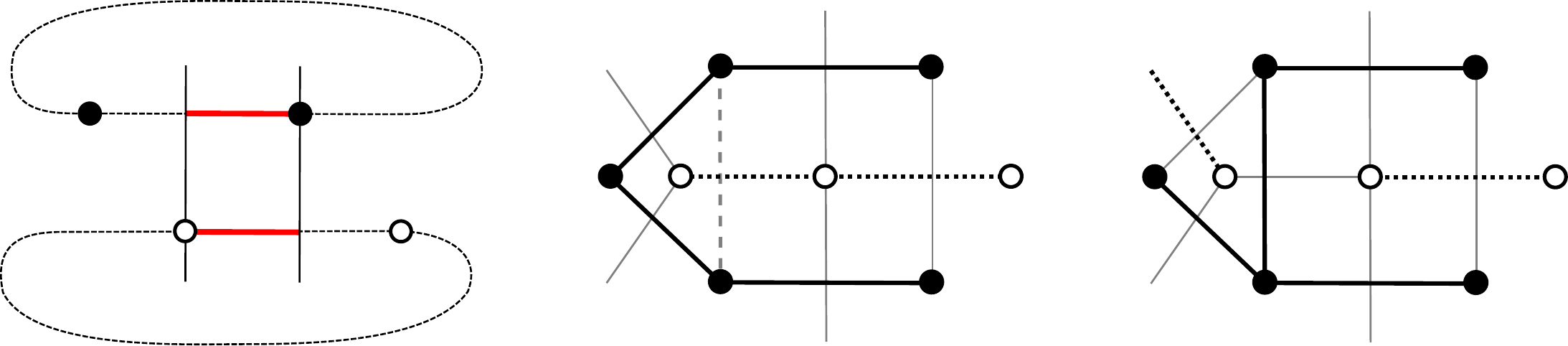}
\caption{On the left: two monochromatic loops (dotted) sharing two opposite edges of a square --up to a click loop move on either one-- lead to a component-reducing clock move of Type III. This is the desired end-state after applying all of the required leaf spins. In the centre: the local configuration that prevents reaching the desired configuration on the left. In particular, the presence of the dashed grey edge causes the problem. We can solve this problem by applying leaf spins in the white graph; the leaves that have to be spun away in the white graph are dotted. The right part shows the effect on the black cycle of spinning away a white leaf.}
\label{fig:bothoncycle}
\end{figure}

Finally, the first black and white cycles are adjacent (as in the left part of Figure \ref{fig:bothoncycle}), and hence we can perform a single clock move to reduce $k$ by one, thus completing the inductive step.
\end{proof}

\begin{remark}
It follows from the proof of Lemma \ref{lem:connectedjordan} that if the perfect admissible matching $x$ has $|J(x)| = 2k+1$, then we can transform it into a dMf by using at most $k$ clock moves of Type III, $k$ click loop moves, and an unspecified number of clock moves of Type I$\slash$II. It is unclear if we can always avoid the use of click path moves.
\end{remark}

\section{Complexes}\label{sec:matchingcplx}

In this section we introduce the matching and Morse complexes -- the objects encapsulating \emph{all} matchings and acyclic matchings, respectively, associated to a given diagram (including both admissible and non-admissible matchings). 

\begin{definition}
The \emph{matching complex of $\Gamma(D)$} is the simplicial complex $\emph{M}(D)$ whose vertices are given by edges in $\Gamma(D)$, and $n$-simplices are matchings with $n+1$ edges.
\end{definition}

\begin{remark}
Alternatively, we can equip the set of partial Kauffman states $\widetilde{X}(D)$ with a simplicial structure as follows: $n+1$ distinct single component partial Kauffman states form an $n$-simplex in $\widetilde{X}(D)$ if and only if their union is still a partial Kauffman state. This is clearly isomorphic to $\text{M}(D)$.
\end{remark}

Another simplicial complex that can be associated to a diagram in this way is the Morse complex of $\Gamma(D)$, first introduced by Chari and Joswig in~\cite{chari2005complexes} for general simplicial complexes. 

\begin{definition}
The \emph{(discrete) Morse complex} $\mathcal{M}(D)$ of $\Gamma(D)$ (or $\Gamma(D,a)$) is the simplicial complex whose vertices are the edges of $\Gamma(D)$ and whose $n$-simplices are spanned by $n+1$ vertices such that the corresponding edges form an acyclic matching in $\Gamma(D)$.
\end{definition}

\begin{remark}
By Remark~\ref{rmk:dmfsubcomplex}, this is a well-defined simplicial complex. In fact, it is clear to see that $\mathcal{M}(D)$ is a subcomplex of $\text{M}(D)$ (similarly for marked diagrams). Note that the matching complex contains both admissible and non-admissible matchings whereas the discrete Morse complex contains only admissible matchings.
\end{remark}

Let $f(D)$ denote the maximal number of arcs in the boundary of a face, maximising over all faces of $S^2 \setminus D$. By \cite[Thm.~3.7]{scoville2020higher}, $\mathcal{M}(D)$ and $\text{M}(D)$ are connected --respectively, simply connected-- if  $4cr(D) \ge 
2f(D) +1 $ and $ 4f(D) +1$, where $cr(D)$ is the \emph{crossing number} of $D$.

It is also possible to give more refined lower bounds on the connectivity of $\text{M}(D)$, using~\cite[Prop.~3.1]{scoville2020higher}: $\text{M}(D)$ is at least $\left( \lfloor \frac{4cr(D)-1}{2f(D)}\rfloor -1\right)$-connected. 
As an example, if $D$ is the diagram in Figure~\ref{fig:tait}, then $\text{M}(D)$ is $1$-connected. The connectivity of the matching complex is also related to the existence of non-extendable pKs, that is, pKs that are not perfect, yet are not faces of simplices of maximal dimension. It is not hard to exhibit examples of non-extendable pKs of arbitrarily high codimension, as illustrated by Figure~\ref{fig:nonextendable}.
\begin{figure}[ht]
\centering
\includegraphics[width=12cm]{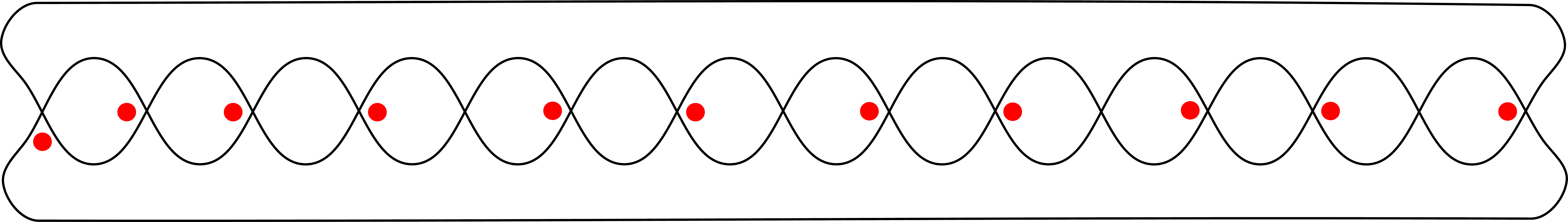}
\caption{An example of an admissible and non-extendable pKs on the diagram of the $T_{15,2}$ torus knot.}
\label{fig:nonextendable}
\end{figure}

Let $\widetilde{X}_{pure}(D)$ denote the set of maximal pKs on the projection $D$, and use the same subscript for the sub-simplicial complexes of $\text{M}(D)$ and $\mathcal{M}(D)$ generated by the top-dimensional simplices.

The pure matching complex of a graph is the union of the pure parts of the matching complexes of all its spanning trees~\cite[Thm.~8]{ayala2008note}. We can prove an analogous result, generalising from graphs to the cellular structures on $S^2$ obtained from Cohen's construction. Denote by $(T,v_b)$ a spanning tree $T \subseteq G_b$ rooted at $v_b$. Recall that a rooted spanning tree in $G_b$ along with a root in $G_w$ uniquely determines a Kauffman state, and hence a perfect acyclic matching in $\Gamma(D)$.

Let $\Delta(T,v_b,v_w)$ denote the simplex in $\mathcal{M}(D)$ spanned by the edges in $\Gamma(D)$ determined by $T$, $v_b$ and $v_w$. Of course, $\Delta(T^\ast,v_w,v_b)=\Delta(T,v_b,v_w)$ since $T$ uniquely determines $T^\ast$ and vice-versa.

\begin{lemma}\label{lem:purepart}
Let $D$ be a diagram with $n+1$ crossings. The simplicial complex $\mathcal{M}_{pure}(D)$ of perfect discrete Morse functions on $\Xi(D)$ is generated by the following set of maximal-dimensional simplices
$$\bigcup_{\substack{T \in Tree(G_b)\\v_w \in V(G_w)\\v_b \in V(G_b)}} \Delta(T,v_b,v_w),$$
where each simplex is of dimension $n$.
\end{lemma}
\begin{proof}
Since each pair $(T,v_b,v_w)$ uniquely identifies a perfect dMf on $\Xi(D)$ (see also \emph{e.g.} Cor. \ref{cor:alldmfsfromKPW}), we get a bijection between the pure part of the Morse complex and the set of simplices above.
\end{proof}

Theorem~\ref{thm:clickclock} gives us insight into the Morse and matching complexes; for example, two perfect dMfs with the same critical cells can differ only by clock moves. At best, they can differ by exactly one clock move, in which case their corresponding simplices in $\mathcal{M}(D)$ are attached along an $n-2$ dimensional face, where $n$ is the dimension of the simplices. If two perfect dMfs differ by one click path move along a click path of length 2 (see Figure~\ref{fig:distanceone}), then the corresponding simplices in $\mathcal{M}(D)$ are attached along an $n-1$ dimensional face. This is unlike the case where we are studying the Morse and matching complexes of graphs (\emph{i.e.}~1-dimensional simplicial complexes), where maximal-dimensional simplices are connected along unions of $(n-2)$-dimensional faces~\cite{ayala2008note}.

\begin{figure}[ht]
  \centering
\includegraphics[width=7cm]{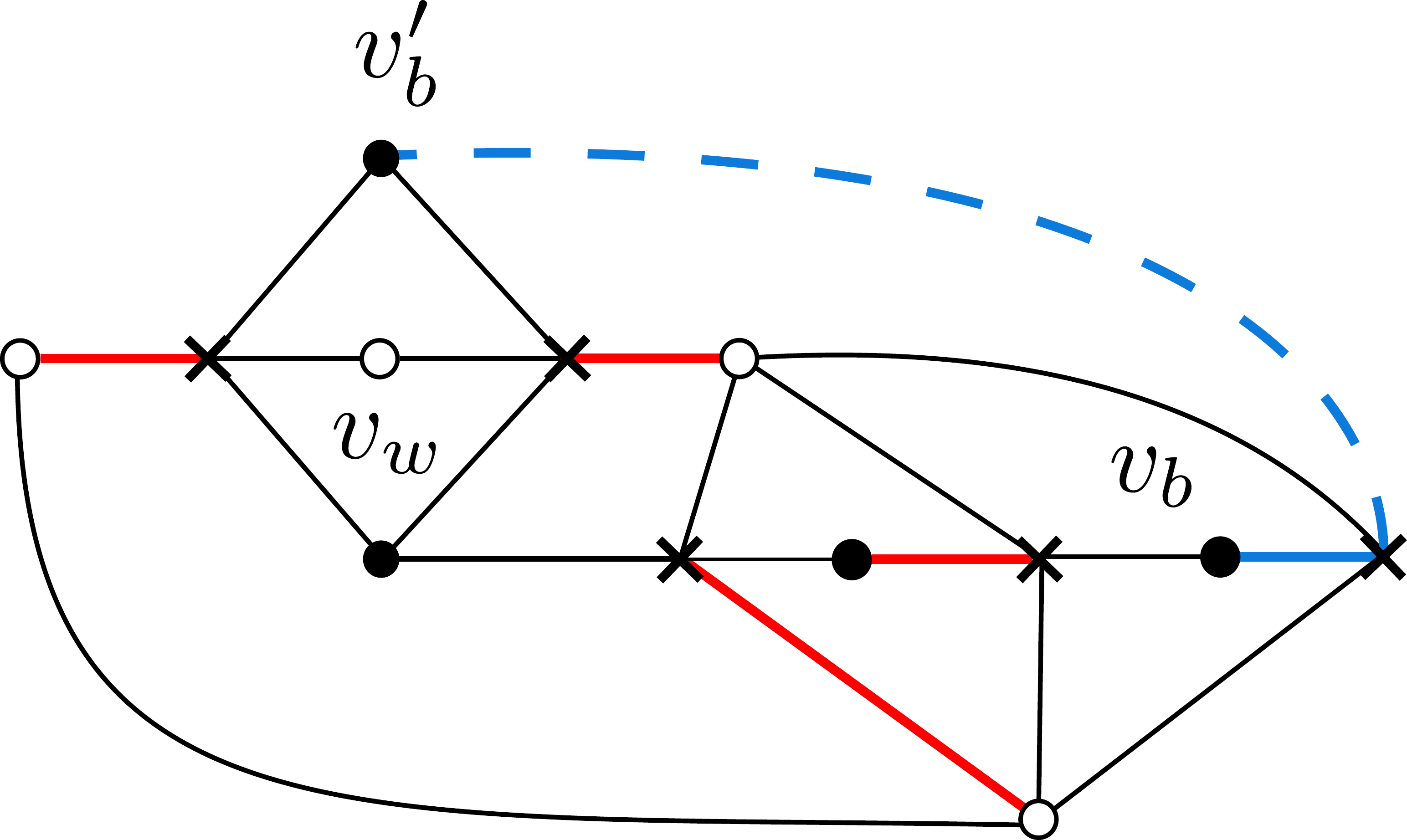}
\caption{The red edges together with the dashed blue one give a perfect dMf with $(v_w, v_b)$ as forbidden regions, while the red edges with the solid blue one give a perfect dMf with forbidden regions  $(v_w,v_b^\prime)$. The union of the two blue edges gives the click path of length 1 between the two. Since the dMfs differ on only one edge, their simplices in $\mathcal{M}_{pure}(D)$ are attached along the $(n-1)$-dimensional face spanned by the red edges, where $n=4$ is the dimension of a maximal simplex in $\mathcal{M}(D)$ since any matching consists of at most 5 edges.}
\label{fig:distanceone}
\end{figure}

Whilst it remains a hard problem to fully understand these complexes (even in just the pure case) the three moves between matchings do reveal some structure. We do point out however that complete knowledge of these moves in $\mathcal{M}(D)$ or $\text{M}(D)$ is not sufficient to determine their homotopy-type(s), indicating that the complexes themselves carry more information than just the set of moves. To see this more concretely, consider the clock graph, which captures all top-dimensional simplices and moves between them for a fixed pair of adjacent critical cells. This graph is not even enough to determine the homotopy-type of the pure Morse complex for adjacent regions, despite being one of the nicest cases dealt with so far, as can be appreciated with the example in Figure~\ref{fig:examplefig8}.

\begin{figure}[ht]
  \centering
\includegraphics[width=9cm]{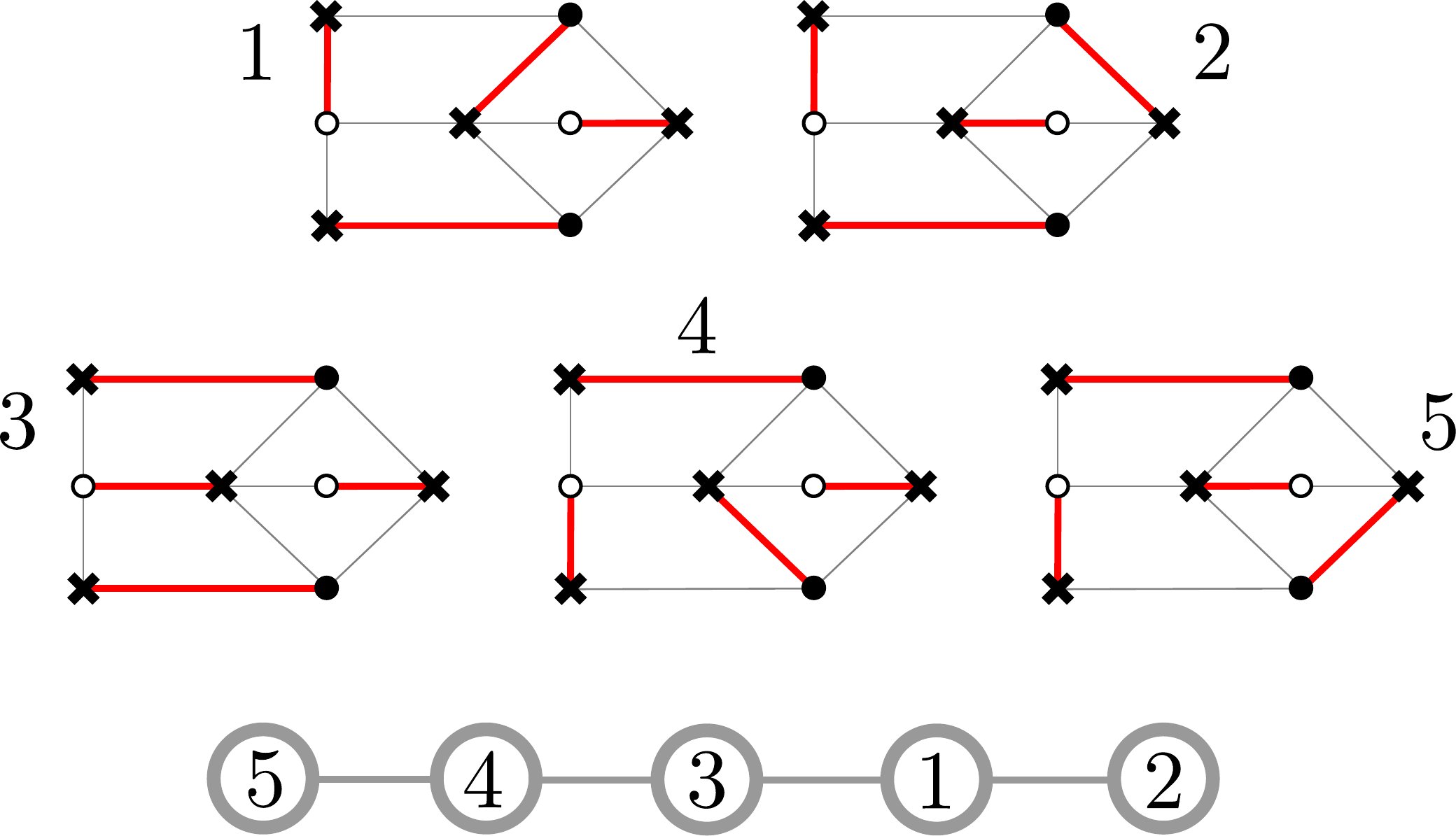}
\caption{These are the $5$ Kauffman states on the projection of the Figure eight knot in Figure~\ref{fig:tait}, with the top-left (red) arc $a$ chosen as the forbidden one. The Clock graph shown below in grey is clearly linear, but it is easy to compute that $\mathcal{M}_{pure}(D,a)$ is a $3$-dimensional simplicial complex  homotopic to a circle.}
\label{fig:examplefig8}
\end{figure}
\bigskip

Before providing the some sample computations of the homology of $\text{M}(D)$ and $\mathcal{M}(D)$ for small diagrams, we make a few remarks;
first we want to highlight a striking difference between dMfs obtained through Cohen's construction, general dMfs and ``regular'' Morse functions; for concreteness let us restrict to the case of a triangulated smooth closed manifold $X$. 

In this case, the matching complex of (the poset graph of) $X$ can be thought of as a discrete analogue of the space of equivalence classes of gradient vector fields of smooth functions on $X$, and $\mathcal{M}(\mathcal{H}(X))$ as the subset of these gradient fields coming from functions that are Morse \cite{benedetti2012smoothing}. It is well-known that smooth Morse functions are dense in the set of smooth functions on $X$, and that ``small'' isotopies are enough to transform any smooth function on $X$ into a Morse one. In the discrete setting, this is no longer true; given a matching, it is generally very hard (see \emph{e.g.}~\cite{nicolaescu2012combinatorial}, the discussion~\cite{86193}, and the end of \cite[Sec. 2]{chari2005complexes}) to find its minimal ``distance'' to a perfect dMf. 

Of course, given any matching one can always obtain a dMf from it by simply removing one edge from each of its monochromatic loops. However, if we restrict to a perfect and admissible matching $x$, it is possible to use the Click-Clock theorem to show that the minimal distance between $x$ and a \emph{perfect} dMf is bounded below by the minimal number of clock and click path\slash loop moves needed to transform $x$ in a dMf, and the proof of the theorem provides us with an almost explicit algorithm that allows us to perform such a transformation. \bigskip

Finally, it would be interesting to tie some of the homological properties of $\text{M}(D)$ and $\mathcal{M}(D)$ to other known graph and knot theoretic invariants and quantities. There is much structure in these simplicial complexes that can be exploited in several ways.
\bigskip

We list here the (reduced) homology of the matching and discrete Morse complexes for all minimal knot diagrams with up to $7$ crossings. The computations are performed using a Sage~\cite{sagemath} program, available upon request.\bigskip

\newpage
\bgroup
\def\arraystretch{2}
\begin{table}[ht]
\begin{center}
\begin{tabular}{| c | Sc || Sc | Sc || Sc | Sc|}
\hline   
    \textbf{Name} &
    \textbf{Minimal Projection} & $\bf{\mathcal{M}(D)}$ &   $\bf{\text{M}(D)}$ &
    $\bf{\mathcal{M}_{pure}(D)}$ &
    $\bf{\text{M}_{pure}(D)}$ \\
\hline 
$3_1$ & \cincludegraphics[scale=0.1] {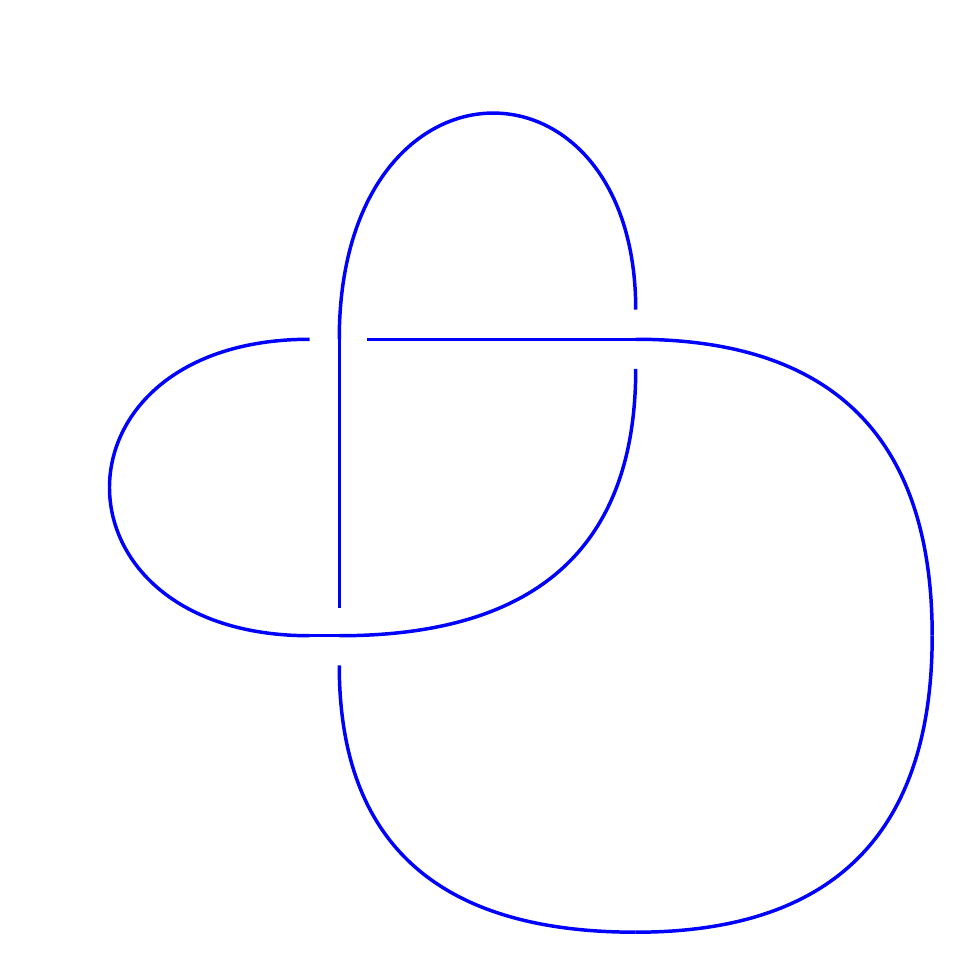}  & $\homol{1}{4}$ & $\homol{2}{4}$ & $\homol{1}{4}$ & $\homol{2}{4}$\\
\hline

$4_1$ & \cincludegraphics[scale=0.1] {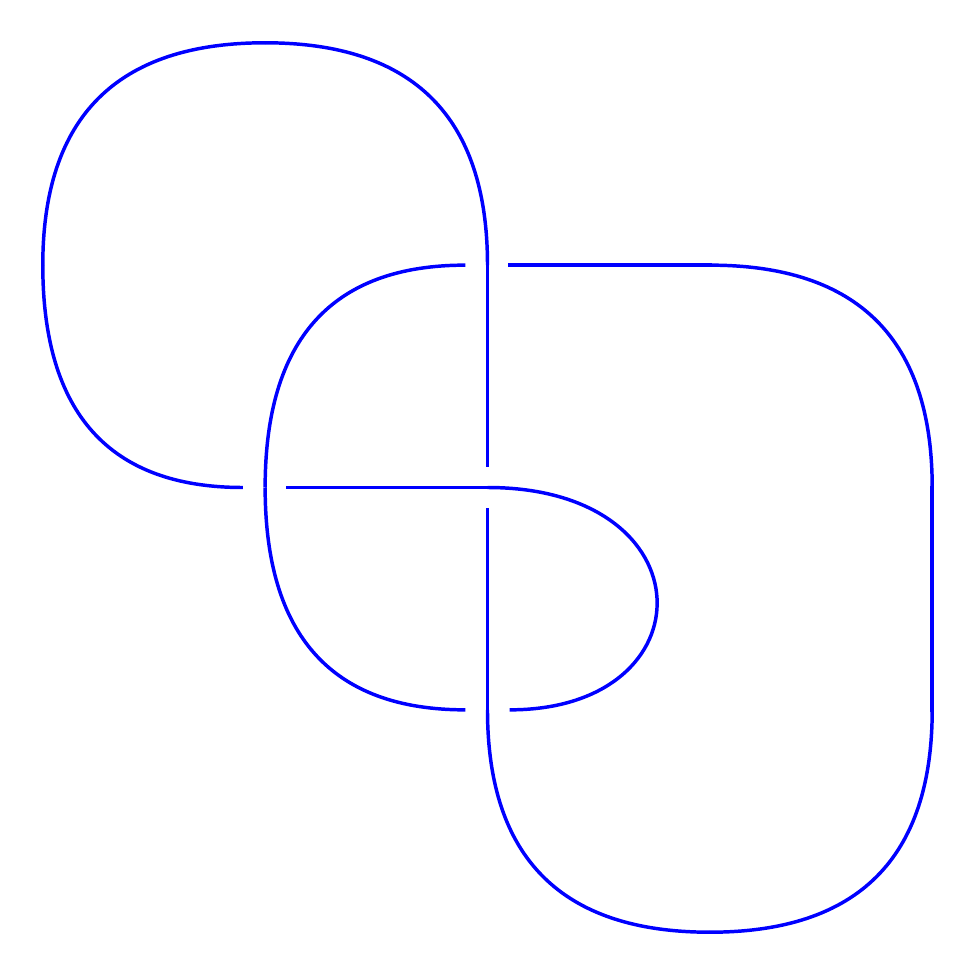} & $\homol{2}{12}$ & $\homol{2}{5} \oplus \Z_{(3)}$ & $\homol{2}{12}$ & $\homol{2}{5} \oplus \Z_{(3)}$\\
\hline

$5_1$ & \cincludegraphics[scale=0.1] {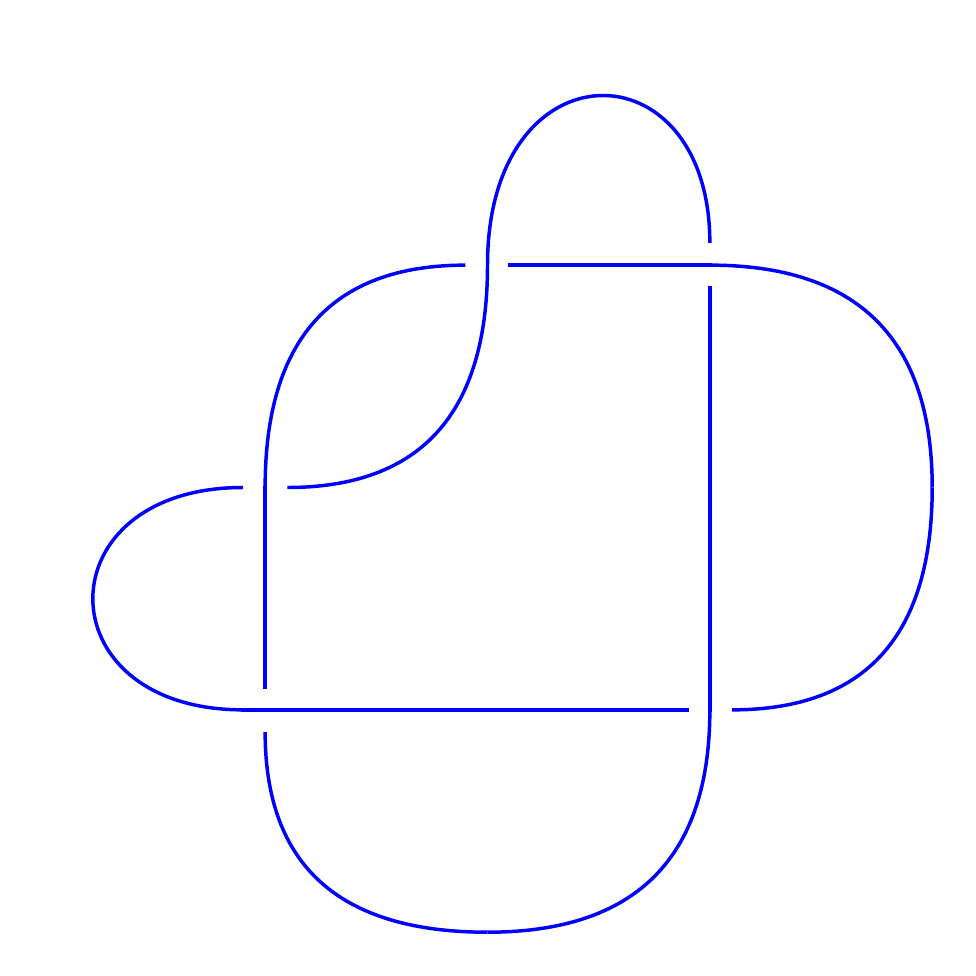} & $\Z_{(2)}\oplus \homol{3}{2}$ & $\homol{3}{19}$ & $\Z_{(2)}\oplus \homol{3}{2}$ & $\homol{3}{9}$\\
\hline

$5_2$ & \cincludegraphics[scale=0.1] {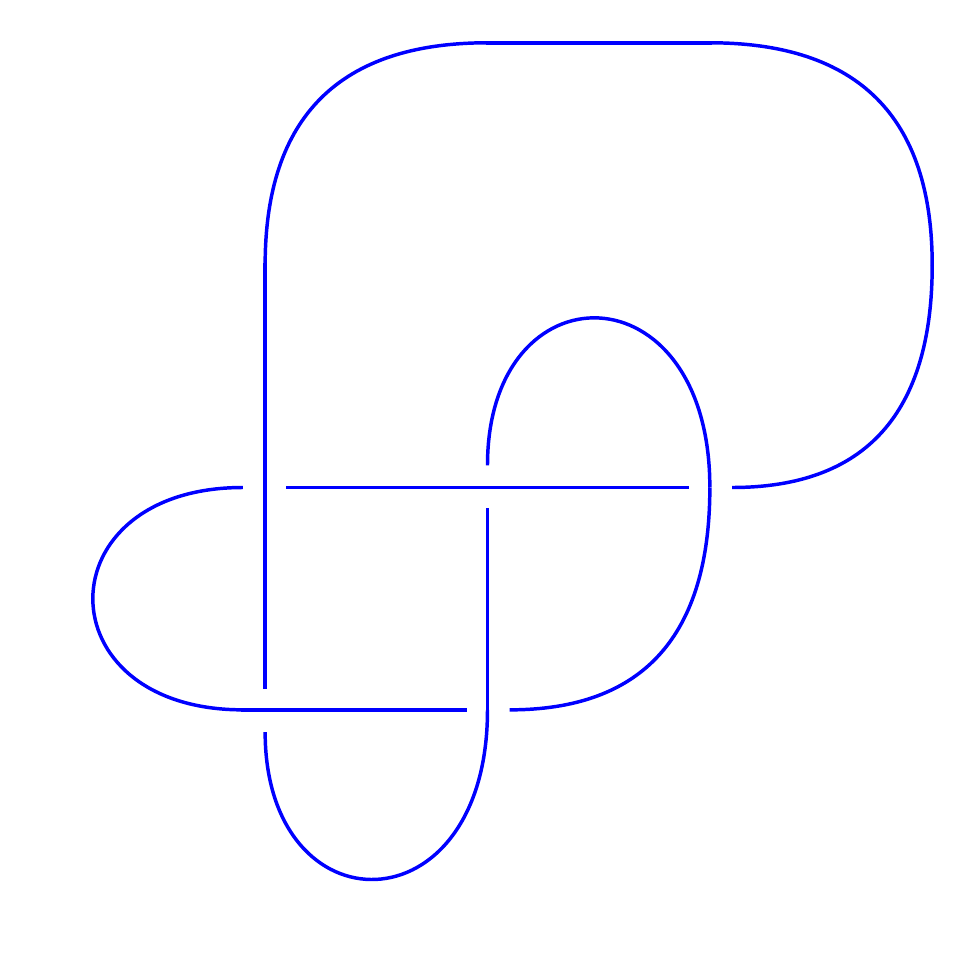} & $\homol{3}{6}$ & $\homol{3}{20}$ & $\homol{3}{6}$ & $\homol{3}{13}$\\
\hline

$6_1$ & \cincludegraphics[scale=0.1] {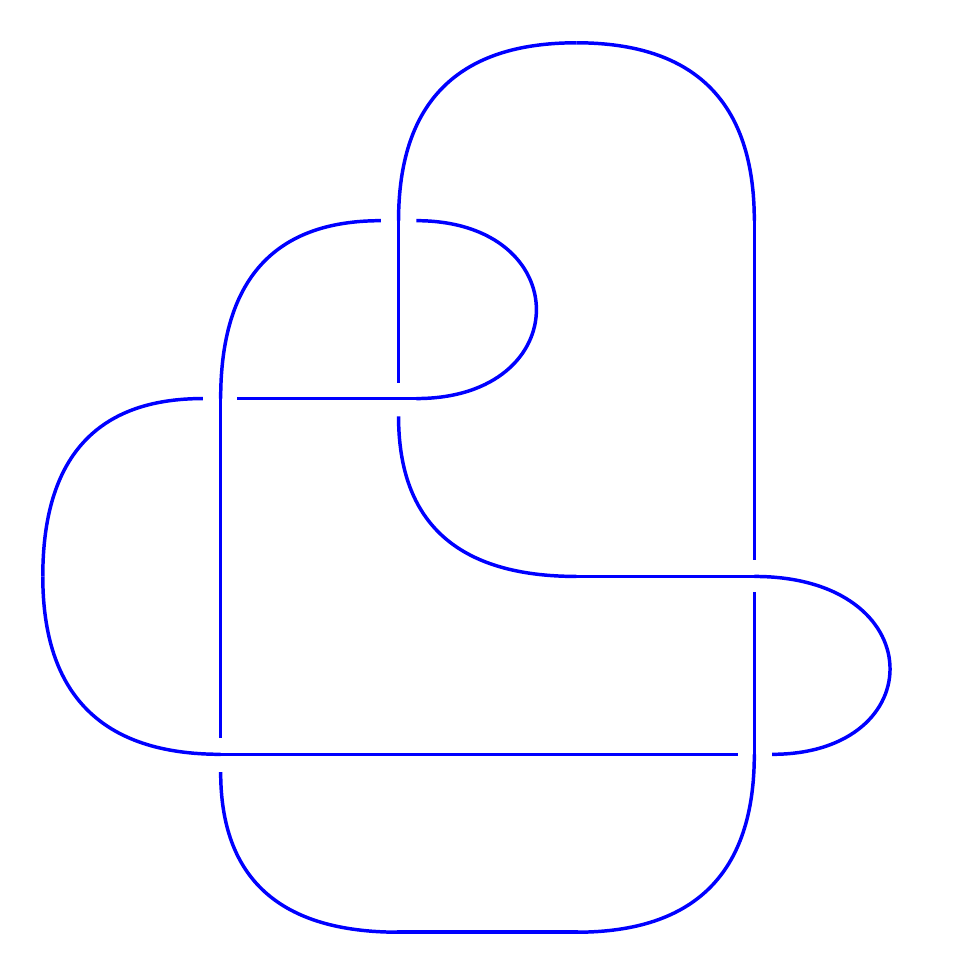} & $\homol{3}{6} \oplus \homol{4}{4}$ & $\homol{4}{28}$ & $\homol{3}{3} \oplus \homol{4}{4}$ & $\homol{3}{2} \oplus \homol{4}{3}$\\
\hline

$6_2$ & \cincludegraphics[scale=0.1] {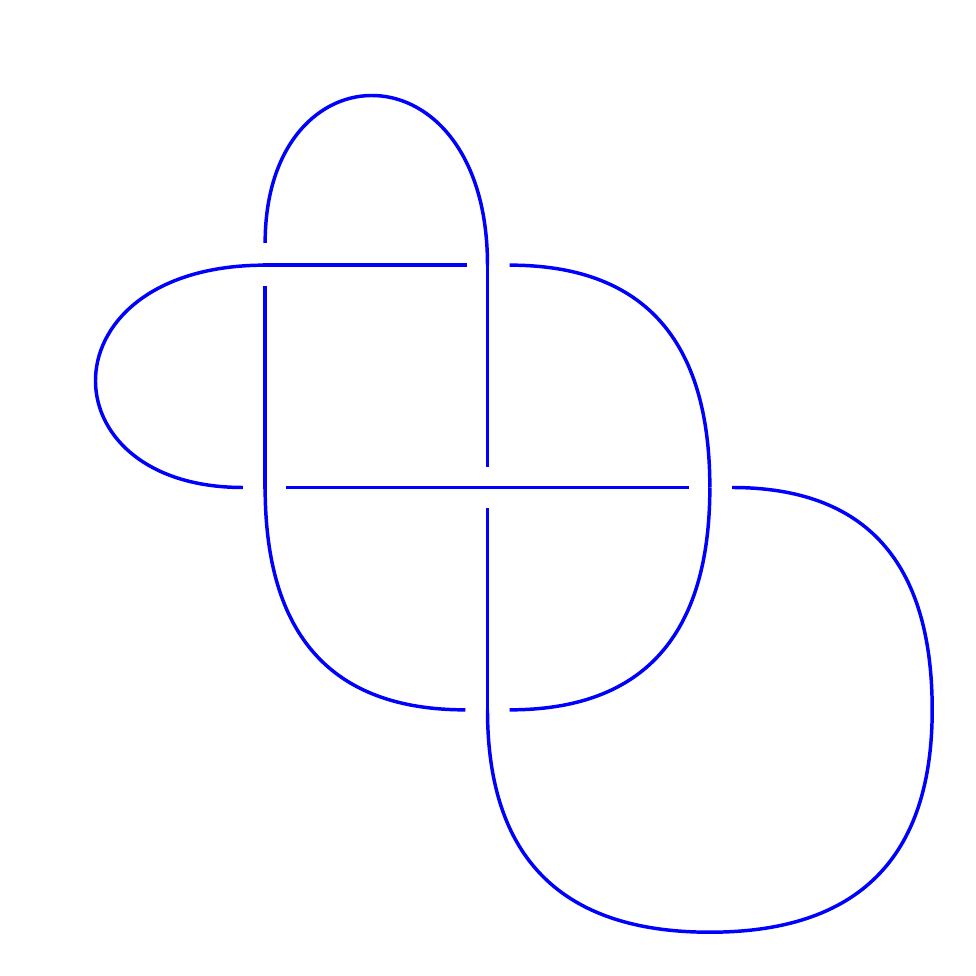} & $\homol{3}{14} \oplus \homol{4}{2}$ & $\homol{4}{30}$ & $\homol{3}{5} \oplus \homol{4}{2}$ & $\homol{3}{2} \oplus \Z_{(4)} $\\
\hline

$6_3$ & \cincludegraphics[scale=0.1] {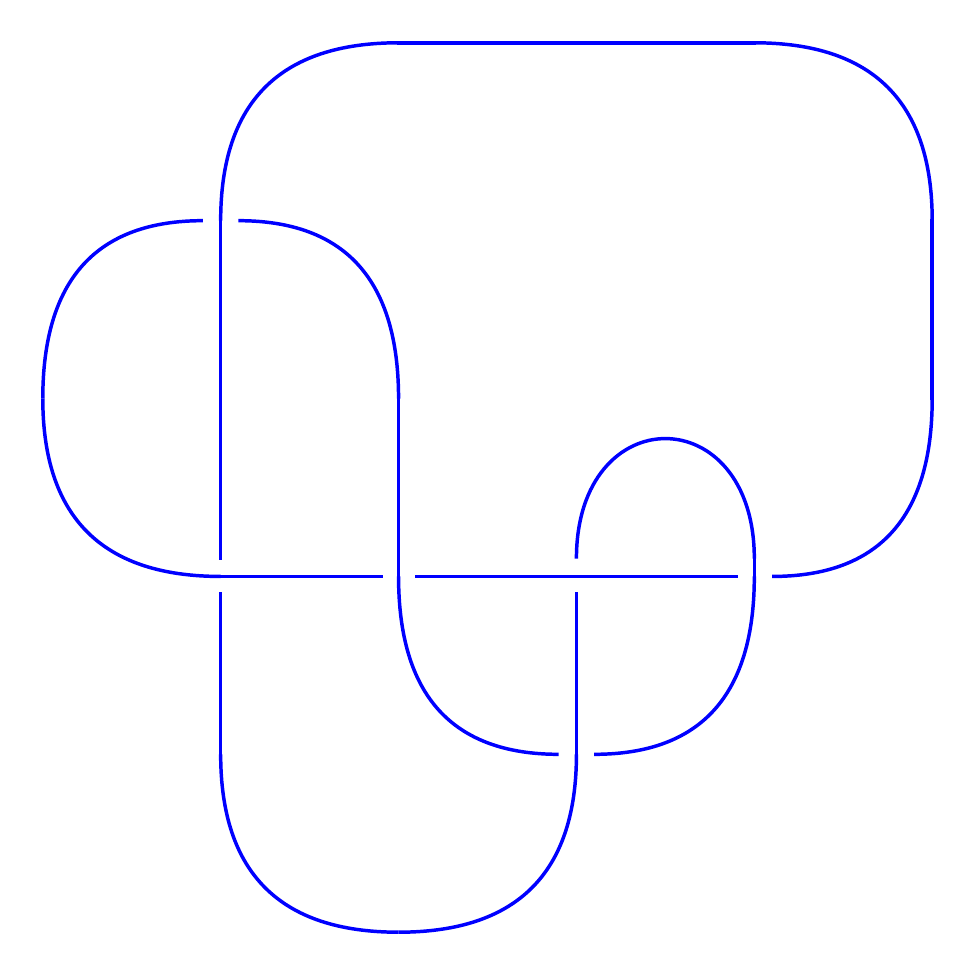} & $\homol{3}{26}$ & $\homol{4}{34}$ & $\homol{3}{32} $ & $\homol{3}{2} \oplus \homol{4}{6}$\\
\hline

$7_1$ & \cincludegraphics[scale=0.1] {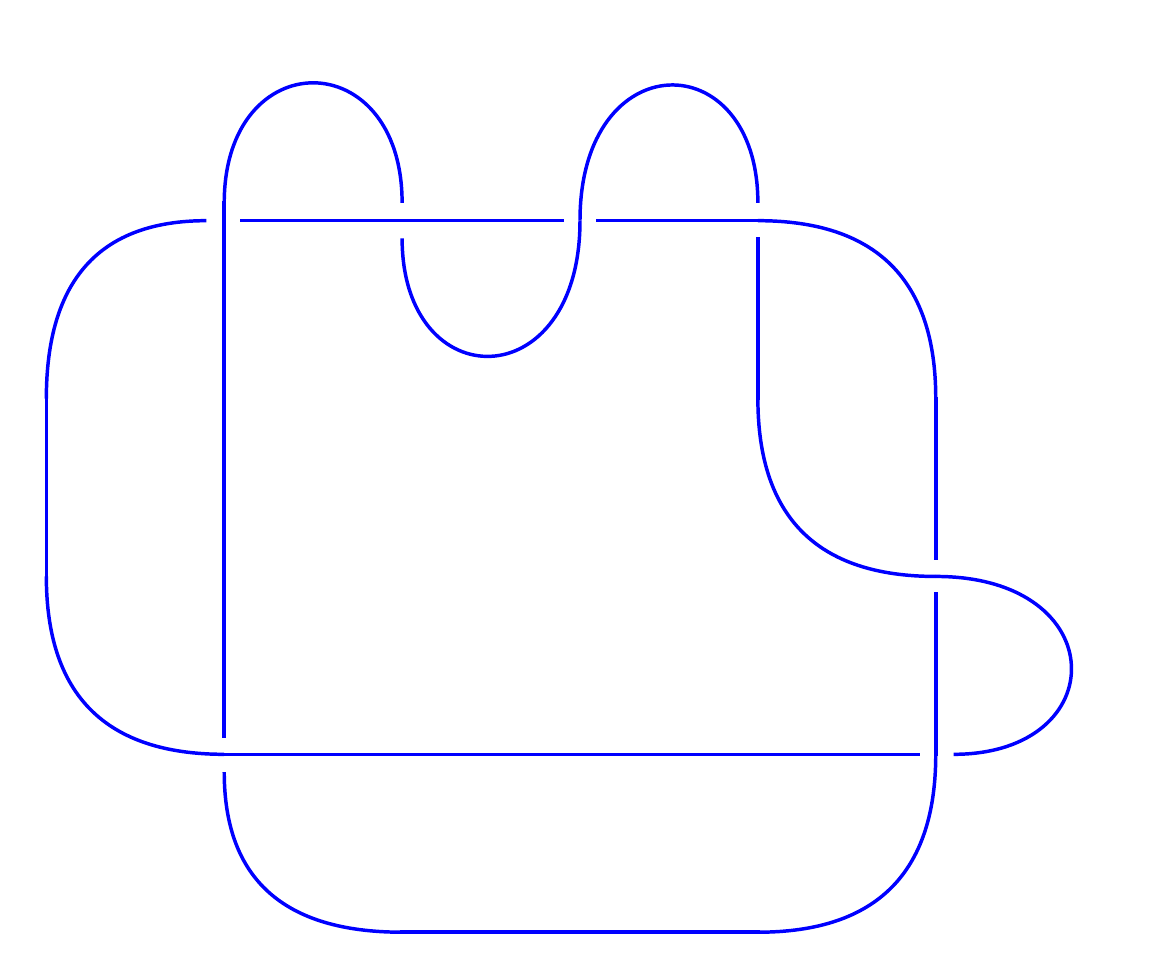} & $\Z_{(2)} \oplus \homol{5}{2}$ & $\homol{4}{2} \oplus \Z_{(5)}$ & $\Z_{(2)} \oplus \homol{5}{2}$ & $\Z_{(4)}$\\
\hline

$7_2$ & \cincludegraphics[scale=0.1] {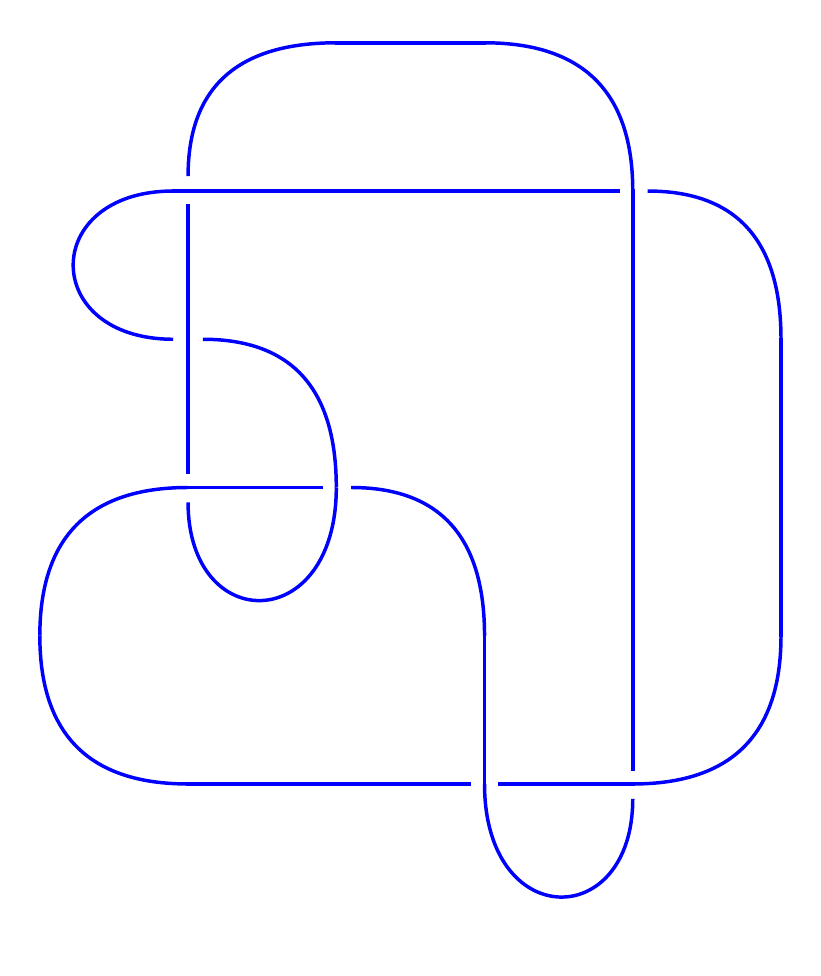} & $\homol{4}{2} \oplus \homol{5}{4}$ & $\homol{4}{2} \oplus \homol{5}{8}$ & $\homol{3}{3} \oplus \homol{5}{4}$ & $\homol{4}{12}$\\
\hline

$7_3$ & \cincludegraphics[scale=0.1] {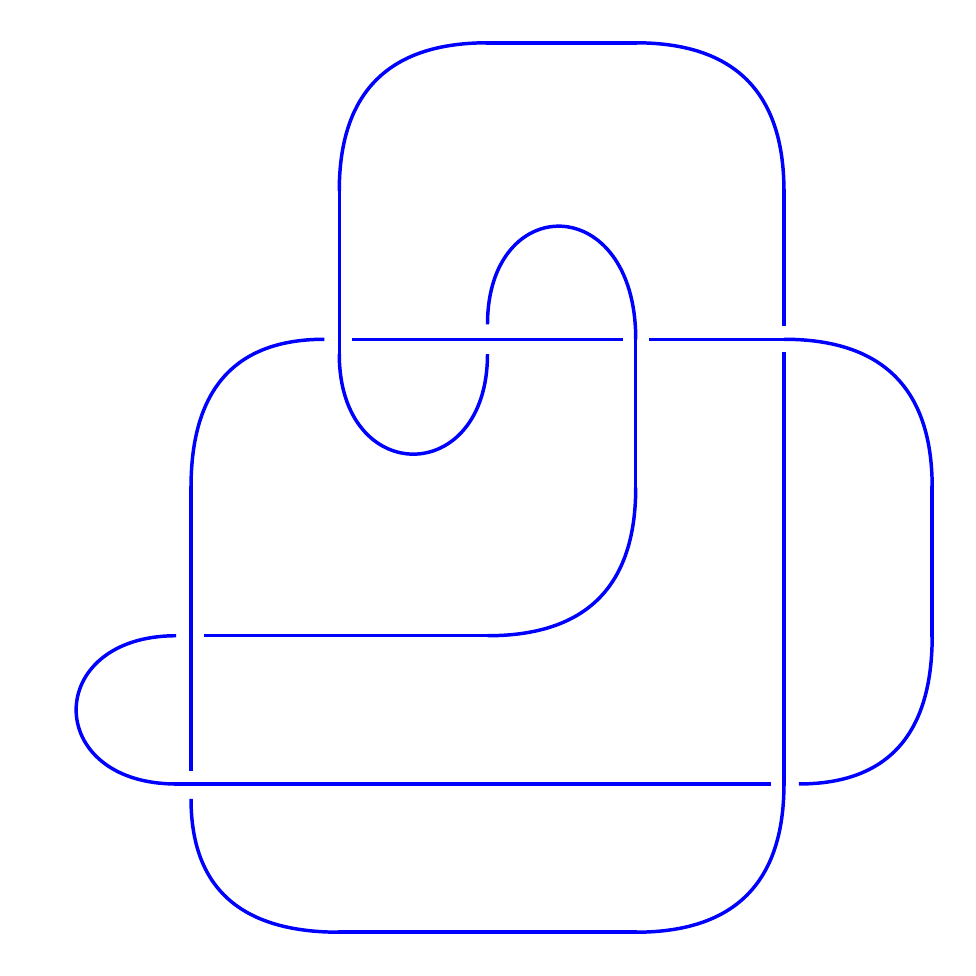} & $\homol{4}{9}$ & $\homol{4}{2} \oplus \homol{5}{3}$ & $\homol{4}{6}$ & $\homol{4}{12}$\\
\hline

$7_4$ & \cincludegraphics[scale=0.1] {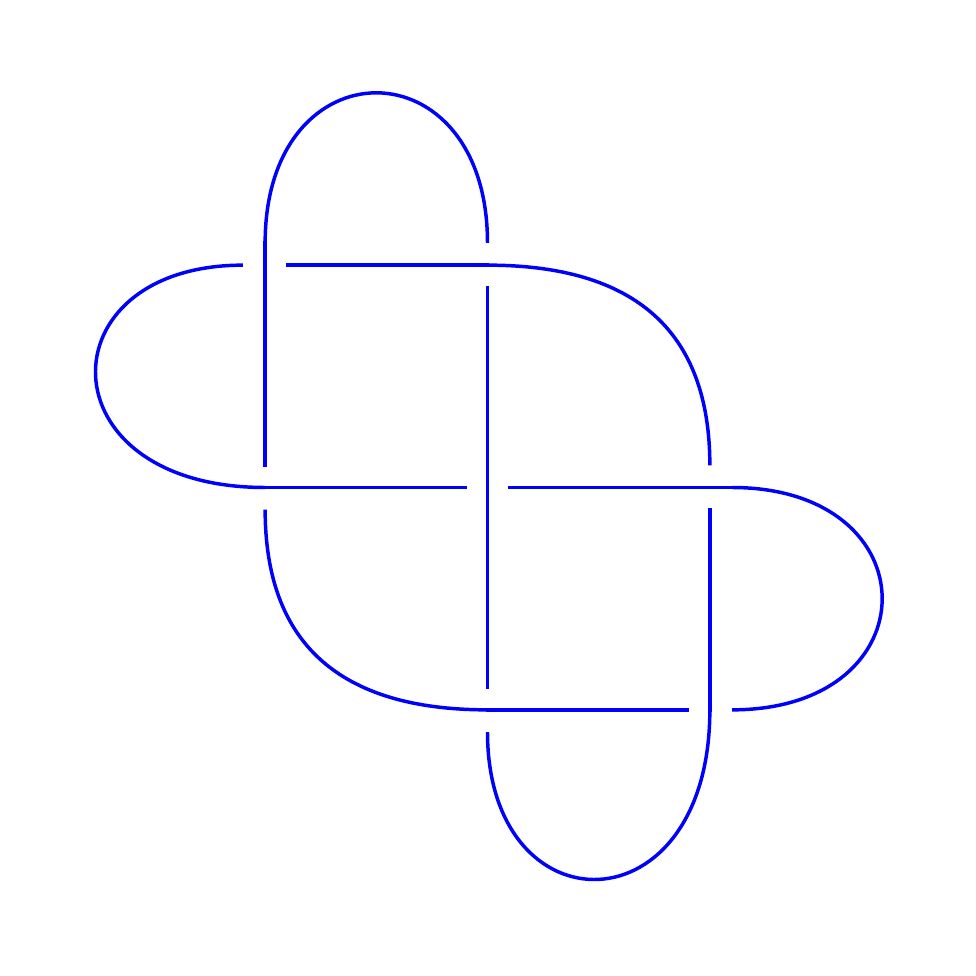} & $\homol{4}{10} \oplus \homol{5}{2}$ & $\homol{4}{2} \oplus \homol{5}{6}$ & $\homol{3}{2} \oplus \homol{5}{2}$ & $\homol{4}{16}$\\
\hline

$7_5$ & \cincludegraphics[scale=0.1] {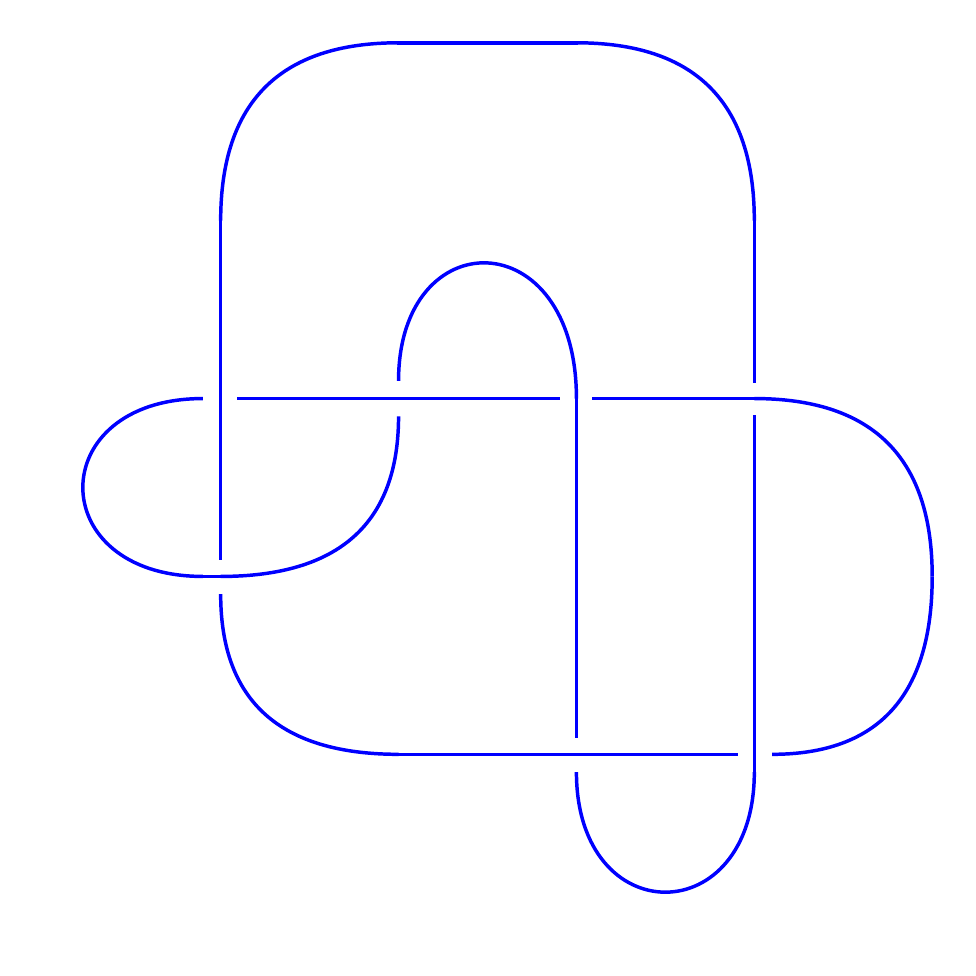} & $\homol{4}{23}$ & $\homol{4}{2} \oplus \homol{5}{9}$ & $\homol{4}{8}$ & $\homol{4}{22}$\\
\hline

$7_6$ & \cincludegraphics[scale=0.1] {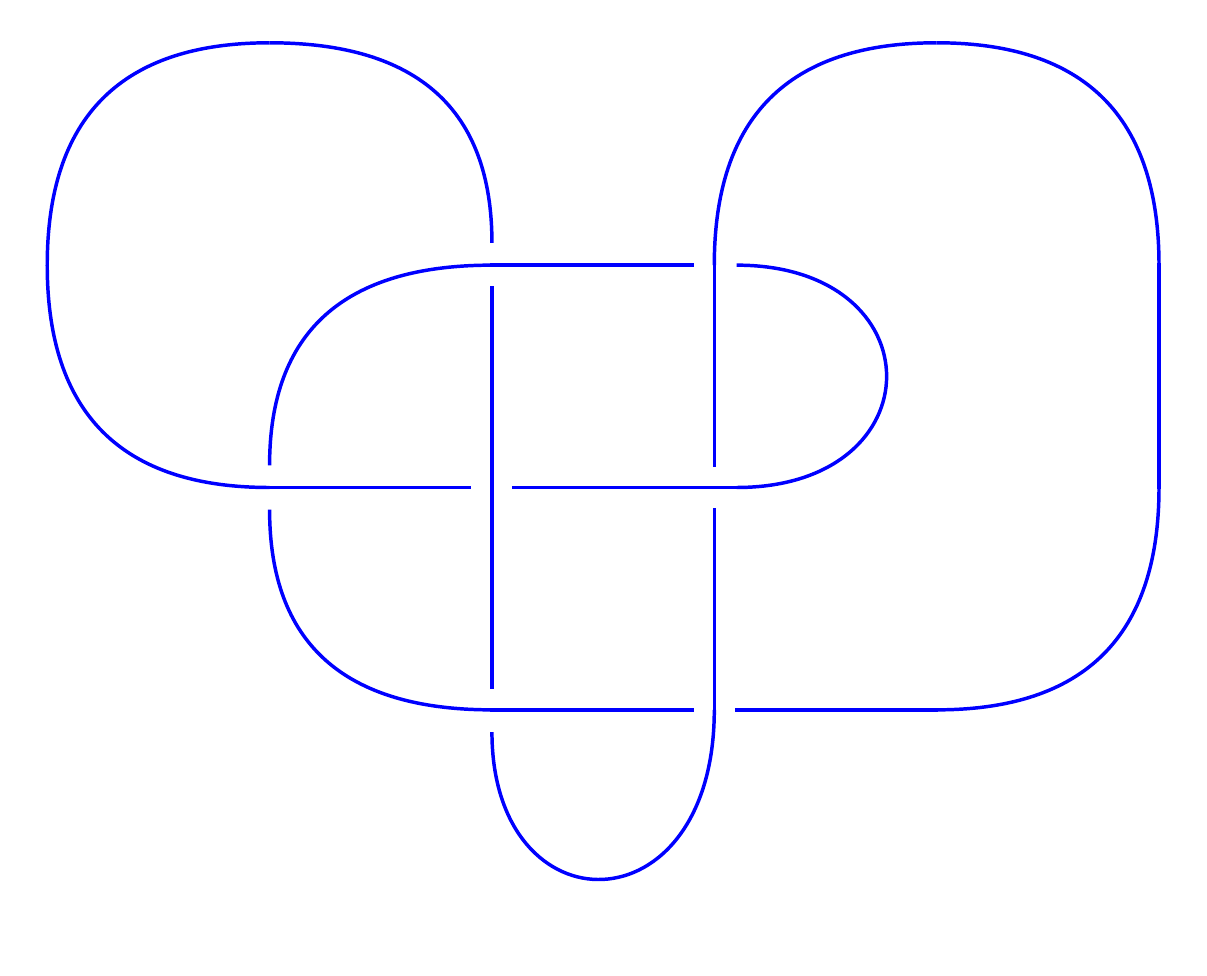} & $\homol{4}{43}$ & $\homol{4}{2} \oplus \homol{5}{14}$ & $\homol{3}{4} \oplus \homol{4}{12}$ & $\homol{4}{26}$\\
\hline

$7_7$ & \cincludegraphics[scale=0.1] {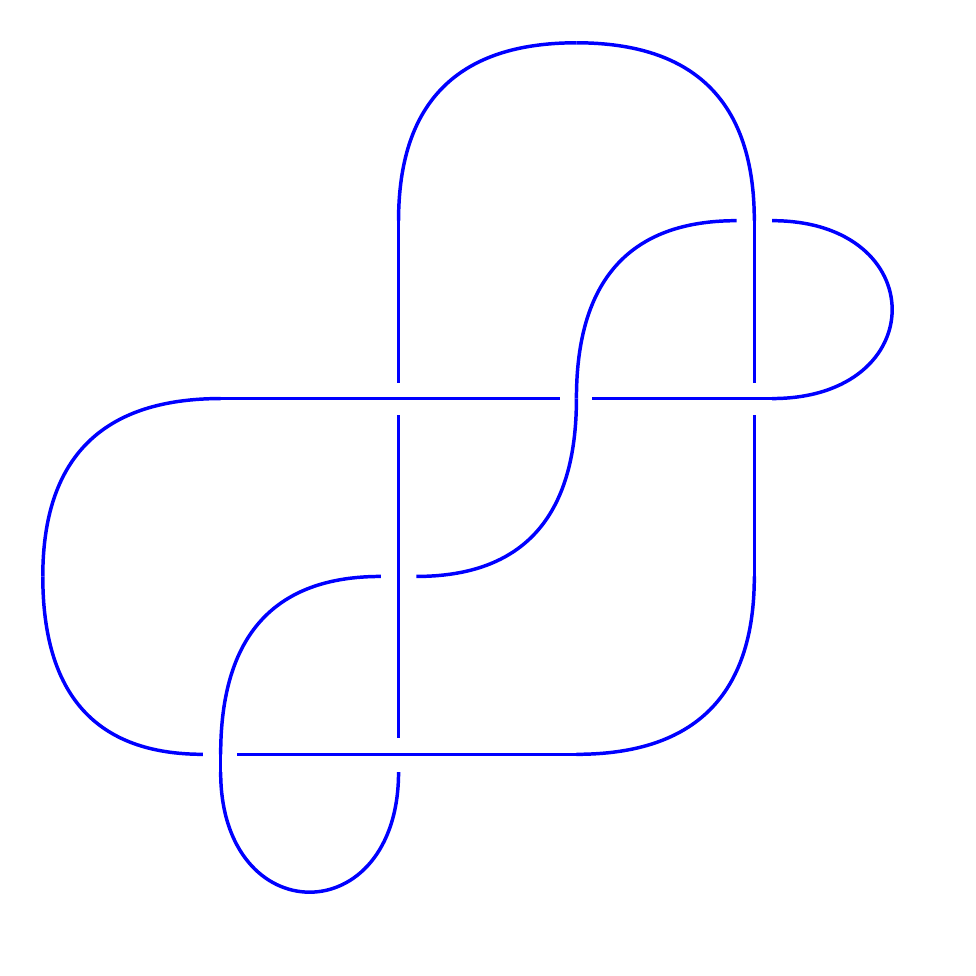} & $\homol{4}{50}$ & $\homol{4}{2} \oplus \homol{5}{14}$ & $\homol{3}{9} \oplus \homol{4}{8}$ & $\homol{4}{30}$\\
\hline
\end{tabular}
\end{center}
\caption{The homology of the various simplicial complexes associated to minimal knot projections, up to $7$ crossings. The notation $\Z^a_{(b)}$ denotes $a$ generators of the homology in degree $b$. The diagrams are generated using SnapPy~\cite{Snappy}.}
\end{table}
\egroup


\end{document}